\documentclass{amsart}
\usepackage{amssymb,amsmath,amsfonts,amsthm,mathrsfs}
\usepackage{enumerate}

\theoremstyle{plain}
\newtheorem{theorem}{Theorem}[section]
\newtheorem{proposition}[theorem]{Proposition}
\newtheorem{lemma}[theorem]{Lemma}

\theoremstyle{definition}
\newtheorem{definition}[theorem]{Definition}
\newtheorem{assumption}[theorem]{Assumption}

\theoremstyle{remark}
\newtheorem{remark}[theorem]{Remark}

\newcommand{\C}{\mathbb{C}}
\newcommand{\R}{\mathbb{R}}
\newcommand{\N}{\mathbb{N}}

\newcommand{\norm}{|\!|\!|}
\renewcommand{\tilde}[1]{\widetilde{#1}}
\renewcommand{\bar}[1]{\overline{#1}}

\newcommand{\tr}{\mathop{\text{\upshape{tr}}}}

\numberwithin{equation}{section}

\begin{document}

\title[Elliptic boundary problem]{An elliptic boundary problem acting on generalized
Sobolev spaces}

\author{R. Denk}
\address{Department of Mathematics and Statistics \\
University of Konstanz \\
D-78457 Konstanz \\
Germany}
\email{robert.denk@uni-konstanz.de}

\author{M.Faierman}
\address{School of Mathematics and Statistics \\
The University of New South Wales \\
UNSW Sydney, NSW 2052 \\
Australia}
\email{m.faiermnan@unsw.edu.au}

\begin{abstract}
We consider an elliptic boundary problem over a bounded region $\Omega$ in $\mathbb{R}^n$ and
acting on the generalized Sobolev  space $W^{0,\chi}_p(\Omega)$ for  $1 < p < \infty$. We note that
similar problems for $\Omega$ either a bounded region in $\mathbb{R}^n$ or a closed manifold
 acting on $W^{0,\chi}_2(\Omega)$, called H\"{o}rmander space, have been the subject of investigation
 by various authors. Then in this paper we will, under the assumption of parameter-ellipticity,
 establish results pertaining to the existence and uniqueness of solutions of the boundary problem.
 Furthermore, under the further assumption that the boundary conditions are null, we will establish results
 pertaining to the spectral properties of the Banach space operator induced by the boundary problem,
 and in particular, to the angular and asymp\-totic distribution of its eigenvalues.
 \end{abstract}

\maketitle

\section{introduction} \label{S:1}
In the latter half of the last century H\"{o}rmander \cite[Chapter II]{H} introduced a class of weight
functions defined on $\mathbb{R}^n$, which he denoted by $\mathcal{K}$ (see Definition \ref{D:2.1} below),
and a Banach space $\mathcal{B}_{p,k}, k \in \mathcal{K}, 1 < p < \infty$, composed of tempered distributions
$u$ such that $\mathscr Fu$ is a measurable function on $\mathbb{R}^n$ and $k\mathscr Fu \in L_p(\mathbb{R}^n)$, where $\mathscr F$ denotes
the Fourier transformation in $\mathbb{R}^n$. He then investigated various properties of this space as well as
the regularity properties of solutions of partial differential equations acting on $\mathcal{B}_{p,k}$. We
might mention at this point  that the space $\mathcal{B}_{2,k}$, called H\"{o}rmander space, is of particular
importance as it gives us a significant generalization of the classical Sobolev space based on $L_2(\mathbb{R}^n)$.

The work of H\"{o}rmander did stimulate significant interest and research at that time, but unlike Sobolev spaces,
the H\"{o}rmander spaces were not widely applied to elliptic boundary problems and to elliptic operators  acting over
closed manifolds. However since the beginning of this century significant investigations have been devoted to these
aforementioned problems (see for example \cite{MM1}, \cite{MM2}, and \cite{AK} as well as the book \cite{MM3}).
Indeed, in the references just cited  the authors restrict themselves to the case $p = 2$ and to a certain subset
of weight functions called interpolation parameters which ensures that every H\"{o}rmander space  based on an
interpolation parameter is actually an interpolation space obtained by interpolating between two Sobolev spaces.
Thus in this way that authors obtain important results pertaining to elliptic boundary problems and to elliptic
operators acting on such H\"{o}rmander spaces defined on closed manifolds.

Shortly after the appearance of the book \cite{H} there appeared the paper of Volevich  and Paneyakh \cite{VP}
presenting, by means of an H\"{o}rmander type weight function, a generalization of Bessel-potential spaces
for $1 < p < \infty$ and then described various properties of this space. This space, which they denote by
$H^{\mu}_p$, is precisely the space of tempered distributions  $u$ such that $\mathscr F^{-1}\,\mu\,\mathscr Fu \in L_p(\mathbb{R}^n)$
 for all $\mu$ belonging to a certain subset, denoted by $\mathcal {K}_0$, of non-vanishing functions in
 $C^{\infty}(\mathbb{R}^n)$ which, together with their inverses, belong to the H\"{o}rmander class of weight
 functions $\mathcal{K}$ and which are multipliers on the Schwartz space $\mathscr S(\mathbb{R}^n)$, that is, as operators
 of multiplication, they map $\mathscr S(\mathbb{R}^n)$ into itself. By defining $\mu\,u(\phi) = u(\mu\,\phi)$ for
 $\mu \in \mathcal{K}_0, \phi \in \mathscr S(\mathbb{R}^n)$  and $u \in \mathscr S^{\prime}(\mathbb{R}^n)$, the
 members of $\mathcal{K}_0$  also become multipliers on the space $\mathscr S^{\prime}(\mathbb{R}^n)$. Lastly, let us
   mention that the spaces obtained by restricting of the members of $H^{\mu}_p$ to subsets of $\mathbb{R}^n$ are also   discussed in \cite{VP}.

  We have mentioned above that $\mathcal{B}_{2,k}$ gives us a generalization of Sobolev spaces  based on $L_2(\mathbb{R}^n)$.
  Motivated by the works cited above, our aim in this paper is to remove the restriction $p = 2$, and by fixing our attention
  upon a certain class of weight functions in $\mathcal{K}$, introduce our generalization of classical Sobolev space based on
  $L_p(\mathbb{R}^n)$, $1   < p <\infty$, as well as on $L_p(G)$ for certain subsets $G$ of $\mathbb{R}^n$. Then we will
  establish various results  pertaining to the operator acting on our generalized Sobolev space  induced by a parameter-elliptic boundary problem.

  Accordingly, we will be concerned here with the boundary problem
\begin{align} \label{E:1.1}
 A(x,D)u(x) - \lambda\,u(x) & = f(x) \; \textrm{for} \; x \in \Omega. \\
 B_j(x,D)u(x) & = g_j(x) \; \textrm{for} \; x \in \Gamma, j = 1,\ldots,m,  \label{E:1.2}
\end{align}
where $\Omega$ is a bounded region in $\mathbb{R}^n$, $n \geq 2$, with boundary $\Gamma$, $A(x,D) = \linebreak
\sum_{|\alpha| \leq  2m} a_{\alpha}(x)D^{\alpha}$
is a linear partial differential operator defined on $\Omega$ of order $2m$, and for $j = 1,\ldots,m$, $ B_j(x,D) = \sum_{|\alpha| \leq m_j}b_{j,\alpha}(x)D^{\alpha}$
is a linear partial differential operator defined on $\Gamma$ of order $m_j < 2m$, while $\lambda \in \mathcal{L}$,
where $\mathcal{L}$ is a closed sector in the complex plane with vertex at the origin.
Our assumptions concerning the boundary
problem \eqref{E:1.1}, \eqref{E:1.2} will be made precise in Section \ref{S:3}.

In Section \ref{S:2} we make precise our definition of the generalized Sobolev space over $\mathbb{R}^n$
and over certain subsets of $\mathbb{R}^n$. This is achieved by firstly defining the subsets
$\mathcal{K}_0$ and $\mathcal{K}_1$ of the H\"{o}rmander class of weight functions $\mathcal{K}$
which will be used in this paper to define the generalized Sobolev spaces with which we will be concerned.
Then for $\chi \in \mathcal{K}_0 \cup \mathcal{K}_1$ we introduce the space $H^{\chi}_p(\mathbb{R}^n)$,
which is a generalization of $L_p(\mathbb{R}^n)$, and describe various properties of this space.
And it is by means of $H^{\chi}_p(\mathbb{R}^n)$ that we are able to introduce the generalized
Sobolev spaces $W^{k,\chi}_p(\mathbb{R}^n),W^{k,\chi}_p(\Omega)$           for $k \in \mathbb{N} \cup \{\,0\,\}$,
and $W^{k-1/p,\chi}_p (\Gamma)$ for $ k \in \mathbb{N}$ (see Definitions \ref{D:2.1}, \ref{D:2.3},
\ref{D:2.8}, \ref{D:2.11}, and \ref{D:2.14}).

In Section \ref{S:3} we make precise our assumptions concerning the boundary problem \eqref{E:1.1}, \eqref{E:1.2}
and then use the results of Section \ref{S:2} to establish our main result concerning
the existence and uniqueness of solutions of this boundary problem (see Theorem \ref{T:3.10} below).

Finally in Section \ref{S:4} we let $A^{\chi}_{B,p}$ denote the Banach space operator, with domain $W^{2m,\chi}_p(\Omega)$,
induced by the boundary problem \eqref{E:1.1}, \eqref{E:1.2} under null boundary conditions. We then prove that
$A^{\chi}_{B,p}$ has compact resolvent and various results are established concerning the completeness
of its principal vectors in $W^{0,\chi}_p(\Omega)$ as well as the angular and asymptotic behaviour of
its eigenvalues (see Theorems \ref{T:4.4}--\ref{T:4.6}).

  \section{Generalized Sobolev space} \label{S:2}
 In this section we are going to introduce our generalization of the classical Sobolev space
 and discuss some of its properties. To this end we need the following terminology.

 Accordingly, we let $x =(x_1,\ldots,x_n) = (x^{\prime},x_n)$ denote a generic point in $\mathbb{R}^n$ and use the notation
 $D_j  = -i\partial/\partial\,x_j, D = (D_1,\ldots,D_n), D^{\alpha}  = D_1^{\alpha_1}\cdots\,D_n^{\alpha_n} = D^{\prime\,\alpha^{\prime}}D_n^{\alpha_n}$, and
 $\xi^{\alpha} = \xi^{\alpha_1}\cdot\ldots\cdot\xi_n^{\alpha_n}$ for $\xi =(\xi_1,\ldots,\xi_n) = (\xi^{\prime},\xi_n) \in \mathbb{R}^n$,
 where $\alpha = (\alpha_1,\ldots,\alpha_n) = (\alpha^{\prime},\alpha_n)$ is a multi-index whose length $\sum_{j=1}^n\alpha_j$ is denoted
 by $|\alpha|$. Differentiation with respect to another variable, say $y \in \mathbb{R}^n$, instead of $x$ will be indicated by replacing
 $D$ and $D^{\alpha}$ by $D_y$ and $D^{\alpha}_y$, respectively.
 We also let $\mathscr S(\mathbb{R}^n)$  denote the Schwartz space of rapidly decreasing functions
on $\mathbb{R}^n$ and let $\mathscr S^{\prime}(\mathbb{R}^n)$ denote its dual, where in this paper it will always be supposed that $\mathscr S^{\prime}(\mathbb{R}^n)$ is equipped
with its weak-$*$-topology.  In addition we let $\langle\xi\rangle = (1+|\xi|^2)^{1/2}$ and $\langle\xi^{\prime}\rangle = (1+|\xi^{\prime}|^2)^{1/2}$
for $\xi = (\xi^{\prime},\xi_n) \in \mathbb{R}^n$, while for $1 < p < \infty, \,0 \leq s < \infty$, and $G$ an open set in $\mathbb{R}^n$, we let
$W^s_p(G)$ denote the Sobolev space of order $s$ related to $L_p(G)$  and denote the norm in this space by $\Vert\cdot\Vert_{s,p,G}$
(see \cite[p.169, p.310, and Theorem 2.3.3, p.177]{T}). Furthermore, we will use norms depending upon the parameter $\lambda \in \mathbb{C} \setminus\{\,0\,\}$,
namely for $k \in \mathbb{N}_0 =         \mathbb{N}\cup \{\,0\,\}$ with $ k \leq 2m$, we let
\begin{equation*}
 \norm u\norm _{k,p,G} = \Vert\,u\,\Vert_{k,p,G} + |\lambda|^{k/2m}\Vert\,u\,\Vert_{0,p,G} \; \textrm{for} \; u \in W^k_p(G).
 \end{equation*}
 We refer to \cite{GK} for details concerning parameter-dependent norms.

 Assume for the moment that when $G \ne \mathbb{R}^n$ the boundary $\partial\,G$ is of class $C^{2m}$. Then for $k \in \mathbb{N}$ with $k \leq 2m$
 the vectors $u \in W_p^k(G)$  have boundary values $v = u\,\bigr|_{\partial\,G}$
 and we denote the space of these boundary values by $W^{k-1/p}_p(\partial\,G)$ and by $\Vert\,\cdot\,\Vert_{k-1/p,p,\partial\,G}$ the norm in this space,
 where $\Vert\,v\Vert_{k-1/p,p,\partial\,G} = \inf {\Vert\,u\,\Vert_{k.p.G}}$ for $v \in W^{k-1/p}_p(\partial\,G)$ and the infimum is taken over
 those $u \in W^k_p(G)$ for which $u\,\big|_{\partial\,G} = v$.
  In addition we will use norms depending upon the parameter
 $\lambda \in \mathbb{C}\,\setminus\, \{\,0\,\}$, namely
 \begin{equation*}
  \norm v\norm _{k-1/p,p,\partial\,G} = \Vert\,v\,\Vert_{k-1/p,p,\partial\,G} + |\lambda|^{(k-1/p)/2m}\Vert\,v\,\Vert_{0,p,\partial\,G} \; \textrm{for} \;
  v \in W^{k-1/p}_p(\partial\,G),
 \end{equation*}
where $\Vert\cdot\Vert_{0,p,\partial\,G}$ denotes the norm in $L_p(\partial\,G)$.   Finally, we let $\mathbb{R}_{\pm} = \{\,t \in \mathbb{R}\,\bigl| t \gtrless 0 \,\}$.

 We are now going to define the generalized Sobolev spaces which will be considered in this paper.
 To this end we require the following definition.
 \begin{definition} \label{D:2.1}
 Let $\mathcal{K}$ denote
 the class of real-valued measurable functions defined on $\mathbb{R}^n$ with values in $(0,\infty)$ such that for each member
 $\chi \in \mathcal{K}$ there exist positive constants $C^{\dagger}_{\chi}$ and $\ell^{\dagger}_{\chi}$ for which the inequality
 \begin{equation*}
 \chi(\xi+\eta) \leq C^{\dagger}_{\chi}\left(1 + |\xi|\right)^{\ell^{\dagger}_{\chi}}\chi(\eta) \; \textrm{holds for every pair}
 \; \xi,\eta \in \mathbb{R}^n.
 \end{equation*}
 Note that
 \begin{equation*}
  (C^{\dagger}_{\chi})^{-1}\chi(0)\left(1+|\xi|\right)^{-\ell^{\dagger}_{\chi}} \leq \chi(\xi) \leq C^{\dagger}_{\chi}\chi(0)\left(1+|\xi|\right)^{\ell^{\dagger}_{\chi}}
  \; \textrm{and also that} \; \chi^{-1} \in \mathcal{K}.
  \end{equation*}
  \end{definition}
  The class $\mathcal{K}$ is precisely the class of weight functions mentioned in Section \ref{S:1} which was introduced by H\"{o}rmander in
  \cite{H} and used there to define the Banach space $\mathcal{B}_{p,k}$.
    As mentioned in Section  \ref{S:1}, Volevich and Paneyakh \cite{VP} defined the more restrictive
    class of    weight functions  $\mathcal{K}_0$  as the set of all smooth functions in $\mathcal K$ which are multipliers in $\mathscr S(\R^n)$ and whose inverses belong to $\mathcal K$, too.
  Our aim now is to use $\mathcal{K}$ in order to define for the case $p \leq 2$ a less restrictive
  generalized Bessel-potential space than that considered in \cite{VP}. However for future considerations we will have to restrict ourselves to the subset $\mathcal{K}_1$
  where
  \[
   \mathcal{K}_1 = \bigl\{\chi \in \mathcal{K}\bigl|\,\chi \in C^{2n^+}(\mathbb{R}^n), |D^{\alpha}\chi(\xi)| \leq C_{\chi}\langle\,\xi\,\rangle^{\ell_{\chi}} \; \textrm{for}
    \; \xi \in \mathbb{R}^n \; \textrm{and}\; |\alpha| \leq  2n^+\,\bigr\},
  \]
    where $C_{\chi}$ and $\ell_{\chi}$ denote positive constants and for $t \geq  0, t^+ = [t/2] +1$, and $[t/2]$ denotes the integer part of $t/2$.

 \begin{remark} \label{R:2.2}
  In order to avoid a proliferation of notation, we will also suppose that for $\chi \in \mathcal{K}_0$, $|D_{\xi}^{\alpha}\chi(\xi)| \leq C_{\chi}\langle\,\xi\,\rangle^{\ell_{\chi}}$
  for $|\alpha| \leq 2n^+$.
  \end{remark}

  We refer to \cite[p.35]{H} and \cite{VP} for examples of function in $\mathcal{K}_0$ and $\mathcal{K}_1$. Note that the following functions
  indicated there: (1) $\chi(\xi) =
\langle\,\xi\,\rangle^t, t \in \mathbb{R}$, (2) $\chi(\xi) = \tilde{P}(\xi) = \left(\sum_{|\alpha| \geq 0}|D_{\xi}^{\alpha}P(\xi)|^2\right)^{1/2}$, where $P$ is a polynomial,
and (3) $\chi(\xi) = \left(1 + \sum_{j=1}^n|\xi_j|^{2\ell_j}\right)^t$, where
$t \in \mathbb{R}$ and the $\ell_j \in \overline{\mathbb{R}_+}$, all  belong to $\mathcal{K}_0$ and $\mathcal{K}_1$. Note also that if $\chi \in \
\mathcal{K}_1$ (resp. $\mathcal{K}_0$), then so does $\chi^{-1}$.

  \begin{definition} \label{D:2.3}
 For $1 < p \leq 2$  we henceforth suppose that $\chi \in \mathcal{K}_1$   and  let
 \begin{equation*}
 \begin{aligned}
  H^{\chi}_p(\mathbb{R}^n) = &\bigl\{\,u \in \mathscr S^{\prime}(\mathbb{R}^n)\bigl|\,\mathscr Fu \; \textrm{is a measurable function on} \; \mathbb{R}^n,
  \chi \mathscr Fu \in \mathscr S^{\prime}(\mathbb{R}^n)   \;\textrm{and} \\
  & \mathscr F^{-1} \chi \mathscr Fu \in L_p(\mathbb{R}^n)\,\bigr\},
  \end{aligned}
  \end{equation*}
while for $2 < p < \infty
$ it will always be supposed that $\chi \in \mathcal{K}_0$, and in this case we let
\begin{equation*}
 H^{\chi}_p(\mathbb{R}^n) = \bigl\{u \in \mathscr S^{\prime}(\mathbb{R}^n)\,\bigl|\, \mathscr F^{-1}\chi \mathscr Fu \in L_p(\mathbb{R}^n)\,\big\}.
 \end{equation*}
 We then equip $H^{\chi}_p(\mathbb{R}^n)$ with the norm   $\Vert\,u\,\Vert^{\chi}_{0,p,\mathbb{R}^n}$ =
 $\Vert\,\mathscr F^{-1}\chi \mathscr Fu\,\Vert_{0,p,\mathbb{R}^n}$ for $ u \in H^{\chi}_p(\mathbb{R}^n)$.
 \end{definition}

 We henceforth suppose that $1 < p < \infty$ and let $\hat{u} = \mathscr Fu$ for $u \in \mathscr S^{\prime}(\mathbb{R}^n)$.
 \begin{proposition} \label{P:2.4}
 $H^{\chi}_p(\mathbb{R}^n)$ is a Banach space.
 \end{proposition}

 \begin{proof}
 Since the proposition is proved in \cite{VP} for the case $p > 2$, we need only prove the proposition for the case $p \leq 2$.
 Accordingly let $\bigl\{\,u_j\,\bigr\}_{j \geq 1}$ denote a Cauchy sequence in $H^{\chi}_p(\mathbb{R}^n)$and put $v_j = \mathscr F^{-1}\chi\,\hat{u}_j$ for $ j \geq 1$.
 Then $\big\{\,v_j\,\}_{j \geq 1}$ is a Cauchy sequence in $L_p(\mathbb{R}^n)$, and hence converges in $L_p(\mathbb{R}^n)$ to some vector $v$.
  It now  follows from the Hausdorff-Young theorem \cite[p.6]{BL} that $\chi\hat{u}_j \to \hat{v}$ in $L_{p^{\prime}}(\mathbb{R}^n)$, where $p^{\prime} = p/(p-1)$,
  and hence also in $\mathscr S^{\prime}(\mathbb{R}^n)$,
as $j \to \infty$.
 Thus for $\phi \in \mathscr S(\mathbb{R}^n)$,
 \[
  \bigl|\left(\hat{u}_j - \chi^{-1}\hat{v}\right)(\phi)\bigr| =  \Bigl|\int_{\mathbb{R}^n}\left(\hat{u}_j - \chi^{-1}\hat{v}\right)\phi\,dx\Bigr|
   \leq \Vert\chi\hat{u}_j - \hat{v}\Vert_{0,p^{\prime},\mathbb{R}^n}\Vert\chi^{-1}\phi\Vert_{0,p,\mathbb{R}^n} \to 0 \]
 as $j\to\infty$.

Thus we have shown that $\hat{u}_j \to \chi^{-1}\hat{v}$ in $\mathscr S^{\prime}(\mathbb{R}^n)$ as $ j \to \infty$,
and hence $u_j \to \mathscr F^{-1}\chi^{-1}\hat{u}$ in $\mathscr S^{\prime}(\mathbb{R}^n)$ as $j \to \infty$, which completes the proof of the proposition.
\end{proof}

 \begin{proposition} \label{P:2.5}
 It is the case that $\mathscr S(\mathbb{R}^n) \subset H^{\chi}_p(\mathbb{R}^n) \subset \mathscr S^{\prime}(\mathbb{R}^n)$ in
 both the algebraic and topological sense. Furthermore, $\mathscr S(\mathbb{R}^n)$ is dense in $H^{\chi}_p(\mathbb{R}^n)$.
 \end{proposition}

 \begin{proof}
 Since the proposition is proved in \cite{VP} for the case $p > 2$, we need only prove the proposition for the case $p \leq 2$.
 Accordingly, it follows from the proof of Proposition \ref{P:2.4} that $H^{\chi}_p(\mathbb{R}^n) \subset \mathscr S^{\prime}(\mathbb{R}^n)$. Turning
 now to $\mathscr S(\mathbb{R}^n)$, we have for $\phi \in \mathscr S(\mathbb{R}^n)$, $\mathscr F^{-1}\chi\hat{\phi} = \mathscr F^{-1}\chi(\xi)\langle\xi\rangle^{-\ell_{\chi}-n^+}\langle\xi\rangle^{\ell_{\chi}+n^+}\hat{\phi}$.
 Hence it follows from Mikhlin's multiplier theorem \cite[p.166]{T} that
 $\Vert\,\mathscr F^{-1}\chi\hat{\phi}\Vert_{0,p,\mathbb{R}^n} \leq C_{\chi,p}\Vert\,\mathscr F^{-1}\langle\xi\rangle^{\ell_{\chi}+n^+}\hat{\phi}\Vert_{0,p,\mathbb{R}^n}$,
 where $C_{\chi,p}$ denotes a positive constant.
 But since $\langle\cdot\rangle^{\ell_{\chi}+n^+}$ is a multiplier on $\mathscr S(\mathbb{R}^n)$, we conclude
 that $\mathscr F^{-1}\langle\xi\rangle^{\ell_{\chi}+n^+}\hat{\phi} \in \mathscr S(\mathbb{R}^n)$, and hence
 $\Vert\,\mathscr F^{-1}\chi\hat{\phi}\Vert_{0,p,\mathbb{R}^n} < \infty$. Thus we conclude that $\mathscr S(\mathbb{R}^n)$ is a subspace of $H^{\chi}_p(\mathbb{R}^n)$.

 Finally let $f$ belong to the dual space of $H^{\chi}_p(\mathbb{R}^n)$. Then $\bigl|f(u)\bigr| \leq C\Vert\,\mathscr v\mathscr F^{-1}\chi\hat{u}\Vert_{0,p,\mathbb{R}^n}$
 for every $u \in H^{\chi}_p(\mathbb{R}^n)$, where $C$ denotes a positive constant, Hence by the Hahn-Banach theorem there exists a $v \in L_{p^{\prime}}(\mathbb{R}^n)$
 such that $f(u) = \int_{\mathbb{R}^n}v\mathscr F^{-1}\chi\hat{u}\,dx$ for $u \in H^{\chi}_p(\mathbb{R}^n)$. This implies that if $\mathscr S(\mathbb{R}^n)$ is not dense in $H^{\chi}_p(\mathbb{R}^n)$,
 then there is a $v \ne 0$ in $L_{p^{\prime}}(\mathbb{R}^n)$ such that

 \begin{equation} \label{E:2.1}
    \int_{\mathbb{R}^n}v\,\mathscr F^{-1}\chi\hat{\phi}\,dx = 0 \; \textrm{for every} \; \phi \in \mathscr S(\mathbb{R}^n).
 \end{equation}

In order to make use of \eqref{E:2.1} to prove our assertion concerning density, we require some further information. To this end let us show that
$\mathscr F^{-1}\chi\mathscr F$ maps $W^{2\ell^+}_p(\mathbb{R}^n)$  continuously into $L_p(\mathbb{R}^n)$, where $\ell =$ max$\{\,\ell_{\chi}+n^+,\ell_{\chi^{-1}}+n^+\,\}$.
Indeed for $u \in W^{2\ell^+}_p(\mathbb{R}^n)$, we have $\Vert\,\mathscr F^{-1}\chi\,\mathscr Fu\Vert_{0,p,\mathbb{R}^n} =
\Vert\,\mathscr F^{-1}\chi\langle\,\cdot\,\rangle^{-2\ell^+}\mathscr F\left((1-\Delta)^{\ell^+}u\right)\Vert_{0,p,\mathbb{R}^n}$, \linebreak where $\Delta$ denotes the Laplacian on $\mathbb{R}^n$,
and hence the required result follows from Mikhlin's multiplier theorem.

Let us also show that $W^{4\ell^+}_p(\mathbb{R}^n) \subset$\, ran\,$\mathscr F^{-1}\chi\,\mathscr F\left(W^{2\ell^+}_p(\mathbb{R}^n)\right)$, where ran denotes range.
Indeed if  $w \in W^{4\ell^+}_p(\mathbb{R}^n)$ and we let $u = \mathscr F^{-1}\chi^{-1}\mathscr F\,w$, then $\Vert\,u\Vert_{2\ell^+,p,\mathbb{R}^n} =
\Vert\,\mathscr F^{-1}\chi^{-1}\langle\,\cdot\,\rangle^{-2\ell^+}\mathscr F\left((1-\Delta)^{2\ell^+}w\right)\Vert_{0,p,\mathbb{R}^n}$, and the required result follows from   \linebreak
Mikhlin's multiplier theorem and the fact that $\mathscr F^{-1}\chi\,\mathscr Fu = w$ (it is important to note that both $\hat{u}$ and $\hat{w}$ are in $L_{p^{\prime}}(\mathbb{R}^n)$).

Next let $w \in W^{4\ell^+}_p(\mathbb{R}^n)$ and let $u \in W^{2\ell^+}_p(\mathbb{R}^n)$ such that $\mathscr F^{-1}\chi\,\mathscr Fu = w$. Also let $\{\,\psi_j\,\}_{j \geq 1}$ denote
a sequence in $\mathscr S(\mathbb{R}^n)$  such that $\psi_j \to u$ in $W^{2\ell^+}_p(\mathbb{R}^n)$. Then for $j \geq 1$, we have
\begin{equation*}
 \int_{\mathbb{R}^n}v\,w\,dx = \int_{\mathbb{R}^n}v\mathscr F^{-1}\chi\,\hat{\psi_j}dx +
 \int_{\mathbb{R}^n}v\mathscr F^{-1}\chi\langle\,\cdot\,\rangle^{-2\ell^+}\mathscr F(1 - \Delta)^{\ell^+}(u - \psi_j)dx,
\end{equation*}
and hence in light of \eqref{E:2.1} and Mikhlin's multiplier theorem we have
\begin{equation*}
\bigl|\int_{\mathbb{R}^n}\,vw\,dx\bigr| \leq C\Vert\,v\Vert_{0,p^{\prime},\mathbb{R}^n}\,\Vert\,u - \psi_j\Vert_{2\tilde{\ell}^+,p,\mathbb{R}^n},
\end{equation*}
where the constant $C$ does not depend upon $j$.
  Thus we conclude that $\int_{\mathbb{R}^n}vw\,dx = 0$ for every $w \in W^{4\ell^+}_p(\mathbb{R}^n)$. But since $ v \in W^{-4\ell^+}_{p^{\prime}}(\mathbb{R}^n)$,
 the dual space of $W^{4\ell^+}_p(\mathbb{R}^n)$ (see \cite[Theorem 2.6, p.198]{T}),
 we must have $v = 0$, which is a contradiction, and this completes the proof of the proposition.
\end{proof}

\begin{proposition} \label{P:2.6}
 $H^{\chi}_p(\mathbb{R}^n)$ is separable.
\end{proposition}
 \noindent\textit{Proof.}
 In light of what was shown above, we see that the embeddings $\mathscr S(\mathbb{R}^n) \subset H_p^{\ell_{\chi}+n^+}(\mathbb{R}^n) \subset H^{\chi}_p(\mathbb{R}^n)$ hold,
 where $H_p^{\ell_{\chi}+n^+}(\mathbb{R}^n)$ denotes the Bessel-potential space of order $\ell_{\chi}+n^+$ based on $L_p(\mathbb{R}^n)$ (see \cite[p.177]{T}).
 Since $H_p^{\ell_{\chi}+n^+}(\mathbb{R}^n)$ is separable, the assertion of the proposition is an immediate consequence
 of Proposition \ref{P:2.4}.
 \hfill{$\square$}

 Under a further assumption on $\chi$ we also have the following result.
 \begin{proposition} \label{P:2.7}
  Suppose   that $u \in H^{\chi}_p(\mathbb{R}^n)$   and $\phi \in C^{\infty}_0(\mathbb{R}^n)$.
  Suppose in addition $\bigl|\xi^{\alpha}D^{\alpha}_{\xi}\left(\chi(\xi+\eta)\chi(\xi)^{-1}\right)\bigr| \leq  c_{\chi}\langle\,\eta\,\rangle^{k_{\chi}}$
  for $\eta \in \mathbb{R}^n$ and $|\alpha| \leq n^+$, where $c_{\chi}$ and $k_{\chi}$ denote positive constants.
  Then $\phi\,u \in H^{\chi}_p(\mathbb{R}^n)$ and
  $\norm \phi\,u\norm _{0,p,\mathbb{R}^n}^{\chi} \leq C_{p,\chi,\phi}\norm u\norm _{0,p,\mathbb{R}^n}^{\chi}$, where the constant $C_{p,\chi,\phi}$ does not depend upon $\lambda$ and $u$.
 \end{proposition}

 \begin{proof}
For $p > 2$ the proposition is proved in \cite{VP}, and hence we restrict
ourselves to the case $p \leq 2$. Accordingly, it is clear that $\phi\,u \in
\mathscr S^{\prime}$, and hence $\mathscr F\phi\,u  \in \mathscr S^{\prime}$. We therefore have to show firstly that $\mathscr F\phi\,u$ is a measurable function on $\mathbb{R}^n$.
Now observe that if we put $\check{f}(x) = f(-x)$, then
\[
 \begin{aligned}
   &\mathscr F_{x \to \xi}\phi\,u = (2\pi)^n\mathscr F^{-1}_{x \to \xi}\check{\phi}\check{u}
 = (2\pi)^n\mathscr F^{-1}_{x \to \xi}\check{\phi} \ast \mathscr F^{-1}_{x \to \xi}\check{u} \\
&= (2\pi)^{-n}\int_{\mathbb{R}^n}\hat{\phi}(\xi-\eta)\hat{u}(\eta)\,d\eta
= (2\pi)^{-n}\int_{\mathbb{R}^n}\hat{\phi}(\xi-\eta)\chi(\eta)^{-1}\chi(\eta)\hat{u}(\eta)\,d\eta.
  \end{aligned}
\]
If we make use of the fact that $\chi\,\hat{u} \in L_{p^{\prime}}(\mathbb{R}^n)$  and appeal to Definition \ref{D:2.1}, then it follows that
$\mathscr F_{x \to \xi}\phi\,u \in L^{loc}_p(\mathbb{R}^n)$, and hence is measurable on $\mathbb{R}^n$. Furthermore, because of density, we need only
complete the remainder of the proof under the assumption that $u \in \mathscr S(\mathbb{R}^n)$. Accordingly, it follows from Fubini's theorem
and Minkowski's inequality that
 \[
 \begin{aligned}
 &\Vert\,\mathscr F^{-1}_{\xi \to x}\chi(\cdot)\mathscr F_{y \to \xi}\phi\,u\Vert_{0,p,\mathbb{R}^n} \\
 &\leq (2\pi)^{-n}\int_{\xi \in \mathbb{R}^n}\langle\,\xi\rangle^{k_{\chi}}|\hat{\phi}(\xi)|\Vert\,\mathscr F^{-1}_{\eta \to x}\langle\,\xi\rangle^{-k_{\chi}}
 \chi(\eta+\xi)\chi(\eta)^{-1}\chi(\eta)\hat{u}\Vert_{0,p,\mathbb{R}^n}\,d\xi,
 \end{aligned}
 \]
 and hence the assertion of the proposition is an immediate consequence of Mikhlin's multiplier theorem.
\end{proof}

 We now turn to the definitions of the generalized Sobolev spaces which will be used in this paper, namely
 $W^{k,\chi}_p(\mathbb{R}^n), W^{k,\chi}_p(\Omega)$ for $k \in \mathbb{N}_0$, and $W^{k-1/p,\chi}_p(\Gamma), k \in \mathbb{N}$, with $k \leq 2m$ in all cases.

 \begin{definition} \label{D:2.8}
 For $k \in \mathbb{N}_0$ let $W^{k,\chi}_p(\mathbb{R}^n) = H^{\langle\,\cdot\,\rangle^k\chi}_p(\mathbb{R}^n)$, and  denote by
 $\Vert\,\cdot\,\Vert^{\chi}_{k,p,\mathbb{R}^n}$ the ordinary norm in this space (see Definition \ref{D:2.3})
 and by $\norm \cdot\norm ^{\chi}_{k,p,\mathbb{R}^n} = \Vert\,\cdot\,\Vert^{\chi}_{k,p,\mathbb{R}^n}
  + |\lambda|^{k/2m}\Vert\,\cdot\Vert^{\chi}_{0,p,\mathbb{R}^n}$ its parameter-dependent norm.
  \end{definition}

 \begin{flushleft}
  Note that $\langle\,\cdot\,\rangle^k\chi $ belongs to $ \mathcal{K}_1$ if $p \leq 2$
   and to  $\mathcal{K}_0$ otherwise.
 \end{flushleft}

 \begin{remark} \label{R:2.9}
  In the sequel it will always be supposed that all function spaces under consideration are equipped with their parameter-dependent norms unless otherwise stated.
  Furthermore, it is to be understood that when not stated explicitly, an isomorphism between any two such spaces is bounded in norm by a constant
  not dependent upon $\lambda$.
 \end{remark}

In the  following proposition  and in the proof of Proposition~\ref{P:2.12} below  we suppose that
for $k \in \N_0$, $W^k_p(\Omega)$ is equipped with its Bessel-potential space norm.

  \begin{proposition} \label{P:2.10}
  Let $k \in \mathbb{N}_0$ with $k \leq 2m$. Then the operator $\mathscr F^{-1}\chi\,\mathscr F$ maps $W^{k,\chi}_p(\mathbb{R}^n)$ isometrically and isomorphically onto $W^k_p(\mathbb{R}^n)$,
  and its norm as well as that of its inverse are bounded by a constant not dependent upon $\lambda$.
 \end{proposition}

 \begin{proof}
 Let $u \in W^{k,\chi}_p(\mathbb{R}^n)$ and let $v = \mathscr F^{-1}\chi\hat{u}$. Then $\Vert\,  \mathscr F^{-1}\langle\,\cdot\,\rangle^k\,\mathscr Fv\Vert_{0,p,\mathbb{R}^n}
 = \Vert\,\mathscr F^{-1}\langle\,\cdot\,\rangle^k\chi\hat{u}\Vert_{0,p,\mathbb{R}^n}$, and hence $v \in W^k_p(\mathbb{R}^n)$ (see \cite[p. 177]{T}).

 Conversely, let $v \in W^k_p(\mathbb{R}^n)$ and let $u = \mathscr F^{-1}\chi^{-1}\hat{v}$. Then $\Vert\,\mathscr F^{-1}\langle\,\cdot\,\rangle^k\,\hat{u}\Vert_{0,p,\mathbb{R}^n} \\
 =
 \Vert\,\mathscr F^{-1}\langle\,\cdot\,\rangle^k\,\hat{v}\Vert_{0,p,\mathbb{R}^n}$,
 and hence $u \in W^{k,\chi}_p(\mathbb{R}^n)$.

 In light of these results and the definitions of the parameter-dependent norms concerned, the proof of the proposition is complete.
\end{proof}

 Let us now turn to the definitions of $W^{k,\chi}_p(\Omega)$ and $W^{k-1/p,\chi}_p(\Gamma)$ for those value of $k$ cited above.
 Accordingly, let $\mathscr D^{\prime}(\Omega)$ denote the space of distributions over $\Omega$.
 \begin{definition} \label{D:2.11}
 Let $W^{k,\chi}_p(\Omega) = \bigl\{\,u \in \mathscr D^{\prime}(\Omega)\;\textrm{such that}\; u
 = \overline{u}\bigl|_{\Omega} \;\textrm{for some}\; \overline{u} \in W^{k,\chi}_p(\mathbb{R}^n)\,\bigr\}$
 and equip $W^{k,\chi}_p(\Omega)$ with the norm $\norm u\norm ^{\chi}_{k,p,\Omega} = \;\textrm{inf}\; \norm \overline{u}\norm ^{\chi}_{k,p,\mathbb{R}^n}$,
 where  the infimum is taken over all $\overline{u} \in W^{k,\chi}_p(\mathbb{R}^n)$ such that $u = \overline{u}\bigl|_{\Omega}$.
 \end{definition}

 Note  that if we let $r_{W^{k,\chi}_p(\mathbb{R}^n) \to W^{k,\chi}_p(\Omega)}$ denote the  operator restricting the
 members of $W^{k,\chi}_p(\mathbb{R}^n)$ to $\Omega$ and $N_{k,p}^{\Omega}$ denote its kernel, then this operator induces a decomposition of $W^{k,\chi}_p(\mathbb{R}^n)$ into equivalent classes
 whereby any two distinct members of $W^{k,\chi}_p(\mathbb{R}^n)$, say $\overline{u}^1$ and $\overline{u}^2$ are said to be equivalent
 if $\overline{u}^1 - \overline{u}^2 \in N_{k,p}^{\Omega}$.
Hence if we denote the induced quotient space by $W^{k,\chi}_p(\mathbb{R}^n)/N_{k,p}^{\Omega}$
and equip it with its quotient space norm, then it is clear that we can identify $W^{k,\chi}_p(\Omega)$ with $W^{k,\chi}_p(\mathbb{R}^n)/N_{k,p}^{\Omega}$
(in the sense that they are isometrically isomorphic to each other). We mention at this point that if $N$ is a subspace of a linear vector space
$Y$ and $X = Y/N$ denotes the corresponding quotient space, then in the sequel we will use the notation $[u]$ to denote the member of $X$ containing $u \in Y$
and $\norm \cdot\norm _X$ to denote the quotient space norm in $X$.

Next let $\mathcal{N}_{k,p}^{\Omega} = \bigl\{\,\overline{v} \in W^k_p(\mathbb{R}^n) \;\textrm{such
that}\; \overline{v} =\mathscr F^{-1}\chi\,\mathscr F\overline{u}\; \textrm{for}\; \overline{u} \in N^{\Omega}_{k,p}\,\bigr\}$
and define the space $W^k_p(\mathbb{R}^n)/\mathcal{N}_{k,p}^{\Omega}$ in an analogous fashion to the way we defined the space $W^{k,\chi}_p(\mathbb{R}^n)/N_{k,p}^{\Omega}$.
 \begin{proposition} \label{P:2.12}
  It is the case that $W^{k,\chi}_p(\Omega)$ is isometrically isomorphic to $W^k_p(\Omega)$.
   \end{proposition}

\begin{proof}
 Since $W^k_p(\Omega)$ is isometrically isomorphic to $W^k_p(\mathbb{R}^n)/\mathcal{N}_{k,p}^{\Omega}$ (equipped with its quotient space
 norm), the proposition will be proved
 if we can show that $W^{k,\chi}_p(\mathbb{R}^n)/N_{k,p}^{\Omega}$ is isometrically isomorphic to $W^k_p(\mathbb{R}^n)/\mathcal{N}_{k.p}^{\Omega}$.
 Accordingly, let $[\overline{u}] \in W^{k,\chi}_p(\mathbb{R}^n)/N_{k,p}^{\Omega}$ and let $\{\,\overline{u}_{\ell}\,\}_1^{\infty}$ be a
 sequence in $W^{k,\chi}_p(\mathbb{R}^n)$
 such that $\overline{u} - \overline{u}_{\ell} \in N_{k,p}^{\Omega}$ and $\lim_{\ell \to \infty}\norm \overline{u}_{\ell}\norm _{k,p,\mathbb{R}^n}^{\chi} =$
 $\norm [\overline{u}]\norm \bigl|_{W^{k,\chi}_p(\mathbb{R}^n)/N_{k,p}^{\Omega}}$.
 Hence if we let $\overline{v} = \mathscr F^{-1}\chi\,\mathscr F\,\overline{u}$ and $\overline{v}_{\ell} = \mathscr F^{-1}\chi\,\mathscr F\,\overline{u}_{\ell}$, then it follows from Proposition \ref{P:2.10}
 that $\norm [\overline{u}]\norm _{W^{k,\chi}_p(\mathbb{R}^n)/N_{k,p}^{\Omega}} =
 \lim_{\ell \to \infty}\norm \overline{v}_{\ell}\norm _{k,p,\mathbb{R}^n} \geq \norm [\overline{v}]\norm _{W^k_p(\mathbb{R}^n)/\mathcal{N}_{k,p}^{\Omega}}$.
 Since similar arguments show that  \\
 $\norm [\overline{u}]\norm _{W^{k,\chi}_p(\mathbb{R}^n)/N_{k,p}^{\Omega}}
 \leq
 \norm [\overline{v}]\norm _{W^k_p(\mathbb{R}^n)/\mathcal{N}_{k,p}^{\Omega}}$,
 the proof of the proposition is complete.
\end{proof}

   Next for $k \in \mathbb{N}, k  \leq 2m,$ let $\gamma_{k,p}^{\dagger}$ (resp. $\overline{\gamma}_{k,p}^{\dagger}$) denote the trace operator mapping $W^k_p(\Omega)$ (resp.$W^k_p(\mathbb{R}^n))$
 onto $W^{k-1/p}_p(\Gamma)$ and let $\mathcal{N}_{k,p}$ (resp. $\overline{\mathcal{N}}_{k,p}$) denote its kernel. Then this operator induces a decomposition
 of $W^k_p(\Omega)$ (resp. $W^k_p(\mathbb{R}^n)$) into equivalent classes whereby any two distinct members of $W^k_p(\Omega)$ (resp. $W^k_p(\mathbb{R}^n)$),
 say $u^1$ and $u^2$ (resp. $\overline{u}^1$ and $\overline{u}^2$) are said to be equivalent
 if $u^1 - u^2 \in \mathcal{N}_{k,p}$ (resp. $\overline{u}^1 - \overline{u}^2 \in \overline{\mathcal{N}}_{k,p}$). Note that if we denote the induced
 quotient space by $W^k_p(\Omega)/\mathcal{N}_{k,p}$ (resp. $W^k_p(\mathbb{R}^n)/\overline{\mathcal{N}}_{k,p}$) and equip it with its quotient space norm,
 then it follows from \cite[Propositions 2.2, 2.3]{ADF} that $W^k_p(\Omega)/\mathcal{N}_{k,p}$ (resp $W^k_p(\mathbb{R}^n)/\overline{\mathcal{N}}_{k,p}$) is isomorphic to
 $W^{k-1/p}_p(\Gamma)$.

  We now denote by $V_{k,p}$ (resp. $\overline{V}_{k,p}$) the operator mapping $W^k_p(\Omega)$ (resp. $W^k_p(\mathbb{R}^n)$) isometrically and isomorphically onto
 $W^{k,\chi}_p(\Omega)$ (resp. $W^{k,\chi}_p(\mathbb{R}^n)$) which is asserted in Proposition \ref{P:2.12} (resp.\,Proposition \ref{P:2.10}) and put
 $N_{k,p} = V_{k,p}\mathcal{N}_{k,p}$ (resp. $\overline{N}_{k,p} = \overline{V}_{k,p}\,\overline{\mathcal{N}}_{k,p}$). Then the decomposition
 of $W^k_p(\Omega)$ (resp. $W^k_p(\mathbb{R}^n)$) into equivalent classes induces a decomposition of $W^{k,\chi}_p(\Omega)$ (resp. $W^{k,\chi}_p(\mathbb{R}^n)$) into equivalent classes
 whereby any two distinct members of $W^{k,\chi}_p(\Omega)$ (resp. $W^{k,\chi}_p(\mathbb{R}^n)$) , say $u^1$ and $u^2$ (resp.
$\overline{u}^1$ and $\overline{u}^2$   are said to be equivalent if $u^1 - u^2 \in N_{k,p}$ (resp. $\overline{u}^1 - \overline{u}^2 \in \overline{N}_{k.p}$).
We denote the induced equivalent space by $W^{k,\chi}_p(\Omega)/N_{k,p}$ (resp. $W^{k,\chi}_p(\mathbb{R}^n)/\overline{N}_{k,p})$ and equip it
with its quotient space norm.

\begin{proposition} \label{P:2.13}
 It is the case that the spaces $W^{k,\chi}_p(\Omega)/N_{k,p},  W^{k,\chi}_p(\mathbb{R}^n)/\overline{N}_{k,p}$, and $W^{k-1/p}_p(\Gamma)$ are isomorphic to each other.
\end{proposition}

\begin{proof}
 In light of what was shown above, it is clear that in order to prove the proposition we need only prove that
$W^{k,\chi}_p(\Omega)/N_{k,p}$ (resp. $W^{k,\chi}_p(\mathbb{R}^n)/\overline{N}_{k,p}$) is isomorphic to $W^k_p(\Omega)/\mathcal{N}_{k,p}$
 (resp.$W^k_p(\mathbb{R}^n)/\overline{\mathcal{N}}_{k.p}$). But for this, we can argue as we did in the proof of Proposition \ref{P:2.12}.
\end{proof}

\begin{definition} \label{D:2.14}
We let $W^{k-1/p,\chi}_p(\Gamma) =  W^{k,\chi}_p(\Omega)/N_{k,p}$ and denote the norm in this space by $\norm \cdot\norm ^{\chi}_{k-1/p.p,\Gamma}$,
where for $[u] \in W^{k-1/p,\chi}_p(\Gamma)$,
\[ \norm [u]\norm ^{\chi}_{k-1/p,p,\Gamma} =
\norm [u]\norm _{W^{k,\chi}_p(\Omega)/N_{k,p}}.\]
\end{definition}

 Finally, in the sequel we denote by $\gamma^{\chi}_{k,p}$ the trace operator mapping $W^{k,\chi}_p(\Omega)$ onto $W_p^{k-1/p,\chi}(\Gamma) = W^{k,\chi}_p(\Omega)/N_{k,p}$.
 We shall also use the symbol $\gamma_{k,p}$ to denote the trace operator mapping $W^k_p(\Omega)$ onto $W^k_p(\Omega)/\mathcal{N}_{k,p}$.

\section{The boundary problem \eqref{E:1.1}, \eqref{E:1.2}} \label{S:3}

In this section we are going to use the results of Section~\ref{S:2} in order to establish our main results concerning the existence and uniqueness of solutions of the boundary problem \eqref{E:1.1}, \eqref{E:1.2}. To this end we require some further information.

\begin{assumption}\label{A:3.1}
  It will henceforth be supposed that
  \begin{enumerate}
    [(1)]
    \item the boundary $\Gamma$ is of class $C^{2(m+(n-1)^++n^+)+1}$,
    \item $a_\alpha\in C^{2(m+(n-1)^++n^+)+1}(\bar\Omega)$ for $|\alpha|\le 2m$, and
    \item $b_{j,\alpha}\in C^{2(m_j^{\#}+(n-1)^+ +n^+)+1}(\Gamma)$ for $|\alpha|\le m_j,\, j=1,\dots,m$ where $m_j^{\#}=m-m_j/2$ if $m_j$ is even and $m_j^{\#}=(2m-m_j)^+$ otherwise.
  \end{enumerate}
\end{assumption}

\begin{remark}
  \label{R:3.2}
  It follows from a standard extension procedure that there is no loss of generality in supposing henceforth that for each $\alpha$ and $j$, $a_\alpha\in C^{2(m+(n-1)^++n^+)+1}(\R^n)$, $b_{j,\alpha}\in C^{2(m_j^{\#}+(n-1)^+ +n^+)+1}(\R^n)$ and have compact support.
\end{remark}

In the sequel we let $\mathring A(x,D)$ (resp., $\mathring B_j(x,D)$) denote the principal part of $A(x,D)$ (resp., $B_j(x,D)$, $j=1,\dots,m$).

\begin{definition}
  \label{D:3.3}
  Let $\mathcal L$ be a closed sector in the complex plane with vertex at the origin. Then we say that the boundary problem \eqref{E:1.1}, \eqref{E:1.2} is parameter-elliptic in $\mathcal L$ if the following two conditions are satisfied:
  \begin{enumerate}
    [(1)]
    \item $\mathring A(x,\xi)-\lambda\not=0$ for $x\in\bar\Omega,\,\xi\in\R^n$, and $\lambda\in\mathcal L$ if $|\xi|+|\lambda|\not=0$;
    \item let $x^0$ be an arbitrary point in $\Gamma$. Assume that the boundary problem \eqref{E:1.1}, \eqref{E:1.2} is rewritten in a local coordinate system associated with $x^0$ wherein $x^0\to 0$ and $\nu\to e_n$, where $\nu$ denotes the interior normal to $\Gamma$ at $x^0$ and $(e_1,\dots,e_n)$ denotes the standard basis in $\R^n$. Then the boundary problem on the half-line
        \begin{align*}
          \mathring A(0,\xi',D_n)v(t) - \lambda v(t) & = 0 \quad\text{for }t=x_n>0,\\
          \mathring B_j(0,\xi',D_n)v(t) & = 0 \quad\text{for }t=0,\, j=1,\dots,m,\\
          v(t)&\to 0\quad\text{as } t\to\infty
        \end{align*}
        has only the trivial solution for $\xi'\in\R^{n-1}$ and $\lambda\in\mathcal L$ if $|\xi'|+|\lambda|\not=0$.
  \end{enumerate}
\end{definition}

We denote by $E$ the strong $\big( 2(m+(n-1)^++n^+)+1\big)$-extension operator mapping $W_p^{2(m+(n-1)^++n^+)+1}(\Omega)$ into $W_p^{2(m+(n-1)^++n^+)+1}(\R^n)$ (see \cite[p.~83]{Ad} for details) and for $v_{2m}\in W_p^{2m}(\Omega)$ let us put $v_{2m}^E = Ev_{2m}$ and $u_{2m}^E = \mathscr F^{-1} \chi^{-1}   \mathscr F v_{2m}^E$. In the following proposition we denote transpose by $\mbox{\;}^\top$.

\begin{proposition}
  \label{P:3.4}
  Suppose that the boundary problem \eqref{E:1.1}, \eqref{E:1.2} is parameter-elliptic in $\mathcal L$. Suppose also that $\lambda\in\mathcal L$ with $|\lambda|\ge \lambda_0>0$, $u\in W_p^{2m,\chi}(\Omega)$, and that $f$ and $g=(g_1,\dots,g_m)^\top$ are defined by \eqref{E:1.1} and \eqref{E:1.2}, respectively. Then $f\in W_p^{0,\chi}(\Omega)$, $g_j\in W_p^{2m-m_j-\frac 1p,\chi}(\Gamma)$ for $j=1,\dots,m$, and
  \[ \norm f\norm_{0,p,\Omega}^\chi + \sum_{j=1}^m \norm g_j\norm_{2m-m_j-\frac 1p,p,\Gamma}^\chi \le C \norm u\norm _{2m,p,\Omega}^\chi,\]
  where the constant  $C$ does not depend upon $u$ and $\lambda$.
\end{proposition}

\begin{remark}
  \label{R:3.5}
  Proposition~\ref{P:3.4} requires some clarifications since we have not yet defined what we mean by $A(x,D)u$ and $B_j(x,D)u$ for $u\in W_p^{2m,\chi}(\Omega)$. Accordingly, in this paper we consider $A(x,D)$ (resp. $B_j(x,D),\, j=1,\dots,m$) as a pseudodifferential operator defined on $\R^n$ with a non-standard symbol $\sum_{|\alpha|\le 2m} a_\alpha(x)\xi^\alpha$ (resp. $\sum_{|\alpha|\le m_j} b_{j,\alpha}(x)\xi^\alpha,\, j=1,\dots,m$). Then for $u\in W_p^{2m,\chi}(\R^n)$ we can appeal to Proposition~\ref{P:2.10} to show that $A(x,D)$ (resp. $B_j(x,D),\, j=1,\dots,m$) can be represented as a pseudodifferential operator defined on $\R^n$ acting on $v^E = \mathscr F^{-1} \chi\mathscr F u^E$. Thus it is by means of these pseudodifferential operators acting on classical Sobolev spaces and the results of Section~\ref{S:2} that enable us in the proof of the proposition to give meaning to the expressions $A(x,D) u = A(x,D) u^E|_{\Omega}$ (resp. $B_j(x,D)u = B_j(x,D)u^E|_{\Omega}$).
\end{remark}

\begin{proof}[Proof of Proposition~\ref{P:3.4}]
In light of what was said in Section~\ref{S:2} we have $u=u_{2m}\in W_p^{2m,\chi}(\Omega)$ and we have to show firstly that $f=(A(x,D)-\lambda)u_{2m} \in W_p^{0,\chi}(\Omega)$, $B_j(x,D)u_{2m} \in W_p^{2m-m_j,\chi}(\Omega)$, $j=1,\dots,m$, and then obtain estimates for $\norm (A(x,D)-\lambda) u_{2m}\norm_{0,p,\Omega}^\chi$ and for $\norm [B_j(x,D)u_{2m}]\norm_{2m-m_j-\frac1p,p,\Gamma}^\chi$.

Accordingly, let us firstly fix our attention upon the operator $A(x,D) =$ \linebreak $\sum_{ |\alpha|\le 2m} a_\alpha(x)D^\alpha$ for $x\in\R^n$, and for a particular $\alpha$ obtain an estimate for
\[ \norm \mathscr F^{-1}_{\xi\to x} \chi(\xi) \mathscr F_{y\to\xi} a_\alpha(y) D_y^\alpha u_{2m}^E\norm_{0,p,\R^n}.\]
To this end (see Remark~\ref{R:3.5}) we consider the operator $a_\alpha(y)D_y^\alpha$ as a pseudodifferential operator defined on $\R^n$ with symbol $\sigma_\alpha(y,\xi) = a_\alpha(y)\xi^\alpha$. Then for our purposes we need to put $\sigma_\alpha(y,\xi)$ in $x$-form (see \cite[p.~141]{G}). To this end we can appeal to \cite[p.~144]{G} to show that in $x$-form $\sigma_\alpha(y,\xi)$ is given (as an oscillatory integral) by
\[ \tilde\sigma_\alpha(x,\xi) = (2\pi)^{-n} \int_{\R^n}\int_{\R^n} e^{-iz\cdot\zeta} \sigma_\alpha(x-z,\xi-\zeta)dzd\zeta,\]
and hence by arguing as in \cite[proof of Theorem~3.1, pp.~23-25]{S} we have $\tilde \sigma_\alpha(x,\xi)= a_\alpha(x) \xi^\alpha + \langle\xi\rangle^{2m-1} \sigma_\alpha^1(x,\xi)$, where
\begin{align*}
  \sigma_\alpha^1(x,\xi) & = \sum_{k=1}^n \int_{\R^n}\int_{\R^n} e^{-i(x-z)\cdot\zeta}(1-\Delta_z)^{m+(n-1)^++n^+} D_{z,k} a_\alpha(z) \sigma_\alpha^2(\xi,\zeta)dzd\zeta,\\
  \sigma_\alpha^2(\xi,\zeta)  & =  \frac i{(2\pi)^n} \int_{t=0}^1 \langle \zeta\rangle^{-2(m+(n-1)^++n^+)} \langle\xi\rangle^{-(2m-1)} D_{\xi,k}(\xi-t\zeta)^\alpha dt,
\end{align*}
and where $\,\cdot\,$ denotes the scalar product, $\Delta$ denotes the Laplacian over $\R^n$, $D_{z,k} = -i\frac\partial{\partial z_k}$, and $D_{\xi,k} = -i\frac{\partial}{\partial \xi_k}$. Thus we see that
\[ \mathscr F_{\xi\to x}^{-1} \chi(\xi) \mathscr F_{y\to \xi} a_\alpha(y) D_y^\alpha u^E_{2m} = I_\alpha^1(x) + I_\alpha^2(x),\]
where $I_\alpha^1(x) = a_\alpha(x) D^\alpha v_{2m}^E(x)$, $I_\alpha^2(x) = \mathscr F_{\xi\to x}^{-1} \sigma_\alpha^1(x,\xi) \mathscr F_{y\to\xi} w(y)$, $w(y) =$ \linebreak $ \mathscr F_{\eta\to y}^{-1} \langle \eta\rangle^{2m-1} \mathscr F_{z\to \eta} v_{2m}^E$, and $v_{2m}^E = \mathscr F^{-1}\chi \mathscr F u_{2m}^E$. It now follows that $\|I_\alpha^1(x)\|_{0,p,\R^n} \le C_1 \|v_{2m}^E\|_{|\alpha|,p,\R^n}$, while it follows from a variant of Mikhlin's multiplier theorem (see \cite[Theorem~1.6]{GK}) that $\|I_\alpha^2(x)\|_{0,p,\R^n} \le C_2 \|v_{2m}^E\|_{2m-1,p,\R^n}$, where the constants $C_j$ do not depend upon $u_{2m}$.

We conclude from these results that
\begin{align*}
  \norm (A(x,D)-\lambda) u_{2m}\norm_{0,p,\Omega}^\chi & = \norm (A(x,D)-\lambda)u_{2m}^E\big|_{\Omega} \norm_{0,p,\Omega}^\chi \le
  \norm (A(x,D)-\lambda) u_{2m}^E\norm _{0,p,\R^n}^\chi \\
  & \le C_3 \norm v_{2m}^E\norm_{0,p,\R^n} \le C_4 \norm v_{2m}\norm_{2m,p,\Omega}
  = C_4 \norm u_{2m}\norm_{2m,p,\Omega}^\chi,
\end{align*}
where the constants $C_j$ do not depend upon $u_{2m}$ and $\lambda$. This proves the assertion concerning $f$.

Suppose next that $1\le j\le m$ and fix our attention upon the operator $B_j(x,D) = \sum_{|\alpha|\le m_j} b_{j,\alpha}(x) D^\alpha$. Then we are now going to show that $B_j(x,D)u_{2m}^E\in $ \linebreak $W_p^{2m-m_j,\chi}(\R^n)$ and obtain an estimate for its norm. To this end let us fix our attention upon the operator $b_{j,\alpha}(x) D^\alpha$ for a particular $\alpha$. Then by arguing with $b_{j,\alpha}(x)D^\alpha$ as we did with $a_\alpha(x)D^\alpha$ above, we can show that
\[ \mathscr F_{\xi\to x}^{-1} \chi(\xi) \mathscr F_{y\to\xi} b_{j,\alpha}(y) D_y^\alpha u_{2m}^E = I_{j,\alpha}^1(x) + I_{j,\alpha}^2(x),\] where
$I_{j,\alpha}^1(x) = b_{j,\alpha}(x) D^\alpha v_{2m}(x)$, $I_{j,\alpha}^2(x) = \mathscr F_{\xi\to x}^{-1} \langle\xi\rangle^{m_j-1} \sigma_{j,\alpha}^1(x,\xi)\mathscr F_{y\to\xi} w(y)$, $w(y) = \mathscr F_{\eta\to y}^{-1} \langle \eta\rangle^{m_j-1} \mathscr F_{z\to\eta} v_{2m}^E$,
\[
   \sigma_{j,\alpha}^1(x,\xi) =
   \sum_{k=1}^n \int_{\R^n}\int_{\R^n} e^{-i(x-z)\cdot \zeta} (1-\Delta_z)^{m_j^\#+(n-1)^++n^+} D_{z,k} b_{j,\alpha}(z) \sigma_{j,\alpha}^2(\xi,\zeta) dz d\zeta,
\]
and
\[ \sigma_{j,\alpha}^2(\xi,\zeta) = \frac{i}{(2\pi)^n} \int_{t=0}^1 \langle \zeta\rangle^{-2(m_j^\#+(n-1)^++n^+)} \langle \xi\rangle^{-(m_j-1)} D_{\xi,k}(\xi-t\zeta)^\alpha dt.\]
It now follows that
\begin{align*}
\norm I_{j,\alpha}^1(x)\norm_{2m-m_j,p,\R^n}& \le C_5 \norm v_{2m}^E\norm_{2m-m_j+|\alpha|, p,\R^n}, \\
\norm I_{j,\alpha}^2(x)\norm_{2m-m_j,p,\R^n} & \le C_6 \norm v^E\norm_{2m-1,p,\R^n},
\end{align*}
and hence that
\[ \norm B_j(x,D) u_{2m}^E\norm_{2m-m_j,p,\R^n}^\chi \le C_7 \norm v_{2m}^E\norm_{2m,p,\R^n} \le C_8 \norm u_{2m}\norm_{2m,p,\Omega}^\chi,\]
where the constants $C_j$ do not depend upon $u_{2m}$ and $\lambda$. On the other hand, we know from Section~\ref{S:2} and \cite[Proposition~2.2]{ADF} that $\norm [B_j(x,D) u_{2m}]\norm_{2m-m_j-1/p,p,\Gamma}^\chi \le C_{9} \norm B_j(x,D) u_{2m}^E\norm_{2m-m_j,p,\R^n}^\chi$, where the constant $C_{9}$ does not depend upon $u$ and $\lambda$. In light of these results, the proof of the proposition is complete.
\end{proof}

A sort of converse to Proposition~\ref{P:3.4} is given by the following proposition.

\begin{proposition}
  \label{P:3.6}
  Suppose that the boundary problem \eqref{E:1.1}, \eqref{E:1.2} is parameter-elliptic in $\mathcal L$. Suppose also that $u\in W_p^{2m,\chi}(\Omega)$ and that $f$ and $g=(g_1,\dots,g_m)^\top$ are defined by \eqref{E:1.1} and \eqref{E:1.2}, respectively. Then $f\in W_p^{0,\chi}(\Omega)$, $g_j\in W_p^{2m-m_j-1/p,\chi}(\Gamma)$ for $j=1,\dots,m$, and there exists a constant $\lambda^0=\lambda^0(p)>0$ such that for $\lambda\in\mathcal L$ with $|\lambda|\ge\lambda^0$, the a priori estimate
  \begin{equation}\label{E:3.1new}
   \norm u\norm_{2m,p,\Omega}^\chi \le C\Big( \norm f\norm_{0,p,\Omega}^\chi + \sum_{j=1}^m \norm g_j\norm_{2m-m_j-1/p,p,\Gamma}^\chi\Big)
   \end{equation}
  holds, where the constant $C$ does not depend upon $u$ and $\lambda$.
\end{proposition}

\begin{proof}
  To begin with we assume that $\lambda\in\mathcal L$ with $|\lambda|\ge \lambda^\dagger$ for some $\lambda^\dagger>0$. Then turning to the proof of Proposition~\ref{P:3.4}, we know from that proof that we have $u=u_{2m}\in W_p^{2m,\chi}(\Omega)$, $f = (A(x,D)-\lambda)u_{2m}\in W_p^{0,\chi}(\Omega)$, and $g_j = [B_j(x,D)u_{2m}]\in W_p^{2m-m_j-1/p,\chi}(\Gamma)$. It was also shown there that with $v_{2m} = V_{2m,p}^{-1} u_{2m}$ (see the text preceding Proposition~\ref{P:2.13}) we have
  \begin{align*}
    \mathscr F_{\xi\to x}^{-1} \chi(\xi) \mathscr F_{y\to\xi} \big(A(y,D_y)-\lambda\big) u_{2m}^E & = \big( A(x,D)-\lambda\big) v_{2m}^E (x) + Qv_{2m}(x),\\
    \mathscr F_{\xi\to x}^{-1}\chi(\xi) \mathscr F_{y\to\xi}B_j(y,D_y)u_{2m}^E & = B_j(x,D)v_{2m}^E(x) + Q_j v_{2m}(x)\quad\text{for } j=1,\dots,m,
  \end{align*}
  where
  \begin{align*}
  Q& =\mathscr F_{\xi\to x}^{-1} \sigma^1(x,\xi)\mathscr F_{y\to\xi}\tilde Q, &&
  Q_j =\mathscr F_{\xi\to x}^{-1} \sigma_j^1(x,\xi)\mathscr F_{y\to\xi} \tilde Q_j, \\
   \sigma^1(x,\xi)   & = \sum_{|\alpha|\le 2m} \sigma_\alpha ^1(x,\xi), && \tilde Qv_{2m} = \mathscr F_{\eta\to y} \langle \eta\rangle^{2m-1} \mathscr F_{z\to\eta} Ev_{2m}, \\
   \text{ and }\sigma_j^1(x,\xi)  & = \sum_{|\alpha|\le m_j} \sigma_{j,\alpha}^1(x,\xi), &&\tilde Q_jv_{2m} = \mathscr F_{\eta\to y}^{-1} \langle \eta \rangle^{m_j-1} \mathscr F_{z\to\eta} Ev_{2m}.
  \end{align*}
  In addition, it follows from the proof of Proposition~\ref{P:3.4} and \cite[Proposition~2.2]{ADF} that
  \begin{align*}
   \norm Qv_{2m}\norm_{0,p,\R^n} + \sum_{j=1}^m \norm Q_j v_{2m}\norm_{2m-m_j,p,\R^n} &  \le C_1 \norm v_{2m}\norm_{2m-1,p,\Omega} \\
   & \le C_2|\lambda|^{-1/(2m)} \norm v\norm_{2m,p,\Omega},
   \end{align*}
  where the constants $C_j$ do not depend upon $u_{2m}$ and $\lambda$.

  Next, let us put $u_0 =(A(x,D)-\lambda)u_{2m}$, $[u_{2m-m_j}] = [B_j(x,D)u_{2m}]$, $j=1,\dots,m$, and denote by $v_0$ (resp. $[v_{2m-m_j}]$) the image of $u_0$ (resp. $[u_{2m-m_j}]$) under the isomorphic mapping of $W_{0,p}^\chi(\Omega)$ onto $W_{0,p}(\Omega)$ (resp. $W_p^{2m-m_j,\chi}(\Omega)/N_{2m-m_j,p}$ onto $W_p^{2m-m_j}(\Omega) / \mathcal N_{2m-m_j,p}$). Let us also denote by $\mathcal P(x,D)-\lambda$ the operator mapping $W_p^{2m}(\Omega)$ into $W_p^0(\Omega)\times\prod_{j=1}^m W_p^{2m-m_j}(\Omega) / \mathcal N_{2m-m_j,p}$ defined by
  \[ \big( \mathcal P(x,D)-\lambda\big) v = \Big\{ (A(x,D)-\lambda)v, \gamma_{2m-m_1,p}B_1(x,D)v,\dots, \gamma_{2m-m_m,p}B_m(x,D)v\Big\}\]
  for $v\in W_p^{2m}(\Omega)$. Then writing $\mathcal P$ for the operator $\mathcal P(x,D)$ we have
  \begin{equation}
    \label{E:3.1}
    (\mathcal P+\tilde{\mathcal P} -\lambda) v_{2m} = (\mathcal P-\lambda)v_{2m} + \tilde{\mathcal P} v_{2m} = \big\{ v_0, [v_{2m-m_1}],\ldots,[v_{2m-m_m}]\big\},
  \end{equation}
  where
  \[ \tilde{\mathcal P} v_{2m} = \big\{ Q_\Omega v_{2m}, \gamma_{2m-m_1,p}(Q_{1,\Omega}v_{2m}),\ldots, \gamma_{2m-m_m,p}(Q_{m,\Omega} v_{2m})\big\}, \]
  $Q_\Omega v_{2m} = Qv_{2m}\big|_{\Omega}$, and $Q_{j,\Omega}v_{2m} = Q_jv_{2m}\big|_{\Omega}$, $j=1,\dots,m$. It is important to observe from what was shown in the previous paragraphs and in the proof of Proposition~\ref{P:3.4} that
  \[ \norm\tilde{\mathcal P}v_{2m}\norm_{W_p^{2m}(\Omega)\to W_p^0(\Omega)\times\prod_{j=1}^m W_p^{2m-m_j}(\Omega)/\mathcal N_{2m-m_j,p}} \le C_3 |\lambda|^{-1/(2m)} \norm v_{2m}\norm_{2m,p,\Omega},\]
  where the constant $C_3$ does not depend upon $u_{2m}$ and $\lambda$. Furthermore, we know from \cite[Theorem~2.1]{ADF} that there exists a constant $\lambda_0 = \lambda_0(p)>0$ such that the set $\{\lambda\in\mathcal L\big|\, |\lambda|\ge\lambda_0\}$ belongs to the resolvent set of $\mathcal P$. Hence if we suppose that $\lambda$ belongs to this set and let $R(\lambda)$ denote the resolvent of $\mathcal P$, then the equation \eqref{E:3.1} can be written as
  \begin{equation}
    \label{E:3.2}
    \big( I+\tilde{\mathcal P} R(\lambda)\big) (\mathcal P-\lambda) v_{2m} = \big\{ v_0,[v_{2m-m_1}],\ldots,[v_{2m-m_m}]\big\}.
  \end{equation}
  On the other hand we know from \cite[Theorem~2.1]{ADF} that
  \[ \norm R(\lambda)\norm_{\{L_p(\Omega), W_p^{2m-m_1}(\Omega)/\mathcal N_{2m-m_1,p},\dots, W_p^{2m-m_m}(\Omega)/\mathcal N_{2m-m_m,p}\}\to W_p^{2m}(\Omega)} \le C_4,\]
  where the constant $C_4$ does not depend upon $\lambda$. Hence if we choose $\lambda^0\ge \lambda^\dagger$ large enough so that $|\lambda|^{-1/(2m)} C_3 C_4\le \frac12$ for $\lambda\in\mathcal L$ with $|\lambda|\ge \lambda^0$, then for $\lambda$ in this set $(I+\tilde{\mathcal P}R(\lambda))$ is a bounded invertible operator on $L_p(\Omega) \times \prod_{j=1}^m W_p^{2m-m_j}(\Omega)/\mathcal N_{2m-m_j,p}$ and the norm of its inverse $(I+\tilde{\mathcal P}R(\lambda))^{-1}$ is bounded by a constant not depending upon $\lambda$. Thus we can write the equation \eqref{E:3.2} in the form
  \begin{equation}
    \label{E:3.3}
    (\mathcal P-\lambda)v_{2m} = \big(I+\tilde{\mathcal P}R(\lambda)\big)^{-1} \big\{ v_0,[v_{2m-m_1}],\ldots, [v_{2m-m_m}]\big\},
  \end{equation}
  and hence it follows from \eqref{E:3.3} and \cite[Theorem~2.1]{ADF} that for $\lambda\in\mathcal L$ with $|\lambda|\ge \lambda^0$ we have
  \begin{equation}
    \label{E:3.4}
    \norm v_{2m}\norm_{2m,p,\Omega} \le C_5\Big( \|v_0\|_{0,p,\Omega} +\sum_{j=1}^m \norm [v_{2m-m_j}]\norm_{2m-m_j-1/p,p,\Gamma}\Big),
  \end{equation}
  where the constant $C_5$ does not depend upon $u_{2m}$ and $\lambda$. The assertion of the proposition is an immediate consequence of this last result and the results of Section~\ref{S:2}.
\end{proof}

We now turn to the question of necessity:

\begin{proposition}
  \label{P:3.7}
  Suppose that for $\lambda\in\mathcal L$ with $|\lambda|\ge\lambda^0$ the a priori estimate \eqref{E:3.1new} holds for every $u\in W_p^{2m,\chi}(\Omega)$, where $f = (A(x,D)-\lambda)u$, $g_j = [B_j(x,D)u]$ for $j=1,\dots, m$, and the constant $C$ does not depend upon $u$ and $\lambda$. Then the boundary problem \eqref{E:1.1}, \eqref{E:1.2} is parameter-elliptic in $\mathcal L$.
\end{proposition}

It is clear from the proofs of Proposition~\ref{P:3.4} and \ref{P:3.6} that Proposition~\ref{P:3.7} will be proved if we can show that the following proposition holds.

\begin{proposition}
  \label{P:3.8} Suppose that for $\lambda\in\mathcal L$ with $|\lambda|\ge \lambda^0$ the a priori estimate \eqref{E:3.4} holds for every $v_{2m}\in W_p^{2m}(\Omega)$, where $v_0 = (A(x,D)-\lambda + Q_\Omega) v_{2m}$, $v_{2m-m_j} = \gamma_{2m-m_j,p} (B_j(x,D)+Q_{j,\Omega})v_{2m}$ for $j=1,\dots,m$, and the constant $C_5$ does not depend upon $v_{2m}$ and $\lambda$. Then the boundary problem \eqref{E:1.1}, \eqref{E:1.2} is parameter-elliptic in $\mathcal L$.
\end{proposition}

In order to proof Proposition~\ref{P:3.8} we need the following lemma.

\begin{lemma}
  \label{L:3.9}
  Let the hypotheses of Proposition~\ref{P:3.8} hold. Then:

  (1) For each point $x^0\in\Omega$ there is a neighbourhood $U\subset\Omega$ of $x^0$ and a number $\lambda_1>0$ such that for every $v_{2m}\in W_p^{2m}(\Omega)$ with support contained in $U$ the estimate
  \[ \norm v_{2m}\norm_{2m,p,\Omega} \le c_1\| (A(x^0,D)-\lambda)v_{2m}\|_{0,p,\Omega}\]
  holds for $\lambda\in\mathcal L$ with $|\lambda|\ge \lambda_1$, where the constant $c_1$ does not depend upon $v_{2m}$ and $\lambda$.

  (2) For each point $x^0\in\Gamma$ there is a neighbourhood $U\subset\R^n$ of $x^0$ and a number $\lambda_2>0$ such that for every $v_{2m}\in W_p^{2m}(\Omega)$ with support contained in $U$ the estimate
  \begin{align*}
   & \norm v_{2m}\norm_{2m,p,\Omega}\le \\
   & \quad c_2\Big( \|(A(x^0,D)-\lambda)v_{2m}\|_{0,p,\Omega} + \sum_{j=1}^m \norm\tilde \gamma_{2m-m_j,p} B_j(x^0,D)v_{2m}\norm_{2m-m_j-1/p,p,\Gamma}\Big)
   \end{align*}
  holds for $\lambda\in\mathcal L$ with $|\lambda|\ge \lambda_2$, where the constant $c_2$ does not depend upon $v_{2m}$ and $\lambda$ and $\tilde\gamma_{2m-m_j,p}$ denotes the trace operator mapping $W_p^{2m-m_j}(\Omega)$ onto $W_p^{2m-m_j-1/p}(\Gamma)$ (see Proposition~\ref{P:2.13}).
\end{lemma}

\begin{proof}
  We know from the proofs of Proposition~\ref{P:3.4} and \ref{P:3.6} that for $\lambda\in\mathcal L$ with $|\lambda|\ge \lambda^0$ and for $x^0\in\Omega$ we have $\|Q_\Omega v_{2m}\|_{0,p,\Omega}\le c_3|\lambda|^{-1/(2m)} \norm v_{2m}\norm_{2m,p,\Omega}$, while for $x^0\in\Gamma$ we have
  \[ \|Q_\Omega v_{2m}\|_{0,p,\Omega} + \sum_{j=1}^m \norm \tilde\gamma_{2m-m_j,p} Q_{j,\Omega} v_{2m}\norm_{2m-m_j-1/p,p,\Gamma}\le c_3 |\lambda|^{-1/(2m)} \norm v_{2m}\norm_{2m,p,\Omega},\]
  where the constant $c_3$ does not depend upon $v_{2m}$ and $\lambda$. Furthermore, for the same values of $\lambda$ we can argue as in \cite[proof of Lemma 4.2]{AV} and appeal to \cite[Proposition 2.2]{ADF} to show that for $x^0\in\Omega$
  \begin{equation}
    \label{E:3.6}
    \| (A(x,D)-A(x^0,D))v_{2m}\|_{0,p,\Omega} \le c' d + c_4|\lambda|^{-1/(2m)} \norm v\norm_{2m,p,\Omega},
  \end{equation}
  while for $x^0\in\Gamma$ we can likewise show that
  \begin{align*}
   & \| (A(x,D)-A(x^0,D))v_{2m}\|_{0,p,\Omega}\\
    & \quad +\sum_{j=1}^m\norm \tilde\gamma_{2m-m_j,p} (B_j(x,D) - B_j(x^0,D))v_{2m}\norm_{2m-m_j-1/p,p,\Gamma}
    \end{align*}
  is bounded by the expression on the right side of \eqref{E:3.6}, where $d$ denotes the diameter of $U$ and the constants $c'$ and $c_4$ do not depend upon $\lambda^0$, $v_{2m}$, and $\lambda$. Hence by choosing $d$ sufficiently small and $\lambda$ sufficiently large, the assertion of the lemma follows immediately.
\end{proof}

\begin{proof}[Proof of Proposition~\ref{P:3.8}]
By appealing to Lemma~\ref{L:3.9} and \cite[Proposition 2.2]{ADF} we can argue as in \cite[proof of Theorem 4.2]{AV} to establish the validity of the proposition.
\end{proof}

We now come to the main result of this section.

\begin{theorem}
  \label{T:3.10}
  Suppose that the boundary problem \eqref{E:1.1}, \eqref{E:1.2} is parameter-elliptic in $\mathcal L$. Then there exists a constant $\lambda^0 = \lambda^0(p)>0$ such that for $\lambda\in\mathcal L$ with $|\lambda|\ge\lambda^0$ the boundary problem \eqref{E:1.1}, \eqref{E:1.2} has a unique solution $u\in W_p^{2m,\chi}(\Omega)$ for every $f\in W_p^{0,\chi}(\Omega)$ and $g=(g_1,\dots,g_m)^\top$ with $g_j\in W_p^{2m-m_j-1/p,\chi}(\Gamma)$, and the a priori estimate
  \[ \norm u\norm_{2m,p,\Omega}^\chi \le C\Big(\norm f\norm_{0,p,\Omega}^\chi + \sum_{j=1}^m \norm g_j\norm_{2m-m_j-1/p,p,\Gamma}^\chi\Big)\]
  holds, where the constant $C$ does not depend upon $f$, the $g_j$, and $\lambda$.
\end{theorem}

\begin{proof}
  We set $u_0:=f$, and we know from Section~\ref{S:2} that  $g_j\in[u_{2m-m_j}]\in W_p^{2m-m_j}(\Omega)/N_{2m-m_j,p}$ for some $u_{2m-m_j}\in W_p^{2m-m_j,\chi}(\Omega)$, $j=1,\dots,m$. Then referring to the proof of Proposition~\ref{P:3.6} for the terminology, let us now seek a solution of the equation
  \begin{equation}
    \label{E:3.7}
    (\mathcal P-\lambda) u_{2m} = \big\{ u_0, [u_{2m-m_1}],\ldots,[u_{2m-m_m}]\big\} \quad\text{for }u_{2m}\in W_p^{2m,\chi}(\Omega)
  \end{equation}
  and for $\lambda\in\mathcal L$ with $|\lambda|\ge\lambda^0$, where $\lambda^0$ is the constant of Proposition~\ref{P:3.6}. Accordingly, let $v_{2m}$ (resp. $v_0$) denote the image of $u_{2m}$ (resp. $u_0$) under the isomorphism mapping $W_p^{2m,\chi}(\Omega)$ onto $W_p^{2m}(\Omega)$ (resp. $W_p^{0,\chi}(\Omega)$ onto $W_p^0(\Omega)$) (see Proposition~\ref{P:2.12}) and let $[v_{2m-m_j}]$ denote the image of $[u_{2m-m_j}]$ under the isomorphism mapping $W_p^{2m-m_j,\chi}(\Omega)/N_{2m-m_j,p}$ onto $W_p^{2m-m_j}(\Omega)/\mathcal N_{2m-m_j,p}$. Then we can argue as we did in the proof of Proposition~\ref{P:3.6} to show that the equation~\eqref{E:3.7} has a solution $u_{2m}$ if and only if the equation
  \begin{equation}
    \label{E:3.8}
    (\mathcal P-\lambda)v_{2m} = (I+\tilde{\mathcal P}R(\lambda))^{-1} \big\{ v_0, [v_{2m-m_j}],\ldots, [v_{2m-m_m}]\big\}
  \end{equation}
  has a solution $v_{2m}\in W_p^{2m}(\Omega)$, where all terms are defined in the proof of Proposition~\ref{P:3.6}.
  Then again we can argue as we did in that proof to show indeed that the equation \eqref{E:3.8} has a unique solution $v_{2m}$ such that the a priori estimate
  \[ \norm v_{2m}\norm_{2m,p,\Omega}\le C\Big( \|v_0\|_{0,p,\Omega}+\sum_{j=1}^m \norm [v_{2m-m_j}]\norm_{2m-m_j-1/p,p,\Gamma}\Big)\]
  holds, where the constant $C$ does not depend upon $v_0$, the $[v_{2m-m_j}]$, and $\lambda$. All the assertions of the theorem follow immediately from these results and those of Section~\ref{S:2}.
\end{proof}

\section{Spectral Theory}\label{S:4}

In this section, we fix our attention upon the boundary problem \eqref{E:1.1}, \eqref{E:1.2} under the assumption that in \eqref{E:1.2} the $g_j$ are all zero, but with one exceptional case in Proposition~\ref{P:4.1} below. Then with $A_{B,p}^\chi$ denoting the Banach space operator induced by this boundary problem, with domain $D(A_{B,p}^\chi)\subset W_p^{2m,\chi}(\Omega)$, we are going to use the results of Sections~\ref{S:2} and \ref{S:3} to show that $A_{B,p}^\chi$ has a compact resolvent and then derive various results pertaining to its spectral properties.

Accordingly, let $\mathcal P^\chi-\lambda$ denote the operator mapping $W_p^{2m,\chi}(\Omega)$ into $W_p^{0,\chi}(\Omega)\times\prod_{j=1}^m W_p^{2m-m_j,\chi}(\Omega)/N_{2m-m_j,p}$ defined by
\[ (\mathcal P^\chi-\lambda)u = \big\{ (A(x,D)-\lambda)u, \gamma_{2m-m_1,p}^\chi B_1(x,D)u,\dots, \gamma_{2m-m_m}^\chi B_m(x,D)u\big\}\]
for $u\in W_p^{2m,\chi}(\Omega)$. Also referring to the proof of Proposition~\ref{P:3.6} for terminology, let $\mathcal P$ (resp. $\mathcal P + \tilde{\mathcal P}$) denote the operator that acts like $\mathcal P(x,D)$ (resp. $\mathcal P(x,D) + \tilde{\mathcal P}$) with domain $W_p^{2m}(\Omega)$ and range in $L_p(\Omega)\times \prod_{j=1}^m W_p^{2m-m_j}(\Omega)/\mathcal N_{2m-m_j,p}$. Then we have shown in the proofs of Proposition~\ref{P:3.6} and Theorem~\ref{T:3.10} that the set $\big\{\lambda\in\mathcal L\big| \,|\lambda|\ge\lambda^0\big\}$ is contained in the resolvent set of each of the operators $\mathcal P^\chi$, $\mathcal P$, and $\mathcal P+\tilde{\mathcal P}$. Now turning for the moment to the general boundary problem \eqref{E:1.1}, \eqref{E:1.2}, that is, without the assumption that the $g_j$ are all zero, we have the following result.

\begin{proposition}
  \label{P:4.1}
  Let $\{f,g_1,\dots,g_m\}\in L_p(\Omega)\times \prod_{j=1}^m W_p^{2m-m_j}(\Omega)/\mathcal N_{2m-m_j,p}$, and for $\lambda\in\mathcal L$ with $|\lambda|\ge\lambda^0$ let $v^1$ (resp. $v^2$) denote the unique vector in $W_p^{2m}(\Omega)$ for which $(\mathcal P-\lambda)v^1 = \{f,g_1,\dots,g_m\}$ (resp. $(\mathcal P+\tilde{\mathcal P}-\lambda)v^2 = \{f,g_1,\dots,g_m\}$. Then $\norm v^1-v^2\norm_{2m,p,\Omega}\le C |\lambda|^{-1/(2m)} \norm v^2\norm_{2m,p,\Omega}$, where the constant $C$ does not depend upon $f$, the $g_j$, and $\lambda$.
\end{proposition}

\begin{proof}
  We have $(\mathcal P-\lambda)(v^1-v^2) = -\{ Q_\Omega v^2, [Q_{1,\Omega} v^2],\dots, [Q_{m,\Omega} v^2]\}$, and hence it follows from the proof of Proposition~\ref{P:3.6} (see also \cite[Theorem~2.1]{ADF}) that \linebreak $\norm v^1-v^2\norm_{2m,p,\Omega}\le C |\lambda|^{-1/(2m)} \norm v^2\norm_{2m,p,\Omega}$, where the constant $C$ is described above. This completes the proof of the Proposition.
\end{proof}

Next let $A_{B,p}^\chi$ denote the operator acting on $W_p^{0,\chi}(\Omega)$ induced by the restriction of the operator $\mathcal P^\chi$ to the set $\{u\in W_p^{2m,\chi}(\Omega)\,\big|\ [B_ju]=0\;\text{for }j=1,\dots,m\}$. Also let $A_{B,p}$ (resp. $\tilde A_{B,p}$) denote the operator acting on $L_p(\Omega)$ induced by the restriction of the operator $\mathcal P$ (resp. $\mathcal P +\tilde{\mathcal P}$) to the set $\{v\in W_p^{2m}(\Omega)\,\big|\, [B_jv]=0,\, j=1,\dots,m\}$ (resp. $\{ v\in W_p^{2m}(\Omega)\,\big|\, [(B_j+Q_{j,\Omega})v] = 0\;\text{for }j=1,\dots,m\}$). Then we know from the proofs of  Proposition~\ref{P:3.6} and Theorem~\ref{T:3.10} that the set $\{\lambda\in\mathcal L\,\big|\, |\lambda|\ge\lambda^0\}$ belongs to the resolvent set of each of the operators $A_{B,p}^\chi$, $A_{B,p}$, and $\tilde A_{B,p}$. Let $R_p^\chi(\lambda)$, $R_p(\lambda)$, and $\tilde R_p(\lambda)$ denote the resolvents of $A_{B,p}^\chi$, $A_{B,p}$, and $\tilde A_{B,p}$, respectively. Then we also know from the above proofs that for $\lambda\in\mathcal L$ with $|\lambda|\ge \lambda^0$
\begin{equation}
  \label{E:4.1}
  \norm R_p^\chi(\lambda)\norm_{W_p^{0,\chi}(\Omega)\to W_p^{2m,\chi}(\Omega)} + \norm R_p(\lambda)\norm_{L_p(\Omega)\to W_p^{2m}(\Omega)} + \norm \tilde R_p(\lambda)\norm_{L_p(\Omega)\to W_p^{2m}(\Omega)} \le C,
\end{equation}
where $C$ denotes a positive constant.

\begin{proposition}
  \label{P:4.2}
  It is the case that $A_{B,p}^\chi$, $A_{B,p}$, and $\tilde A_{B,p}$ have compact resolvents. Furthermore,  $\mu$ is an eigenvalue of $A_{B,p}^\chi$ of algebraic multiplicity $k$ if and only if $\mu$ is an eigenvalue of $\tilde A_{B,p}$ of algebraic multiplicity $k$.
\end{proposition}

\begin{proof} Suppose that $\lambda\in\mathcal L$ with $|\lambda|\ge\lambda^0$. Then it follows from \eqref{E:4.1}, the fact that the embedding of $W_p^{2m}(\Omega)$ into $L_p(\Omega)$ is compact (see \cite[p.~144]{Ad}), and from the resolvent equation (see \cite[p.~36]{K}) that the first assertion of the proposition is true for $A_{B,p}$ and $\tilde A_{B,p}$. That this is also true for the operator $A_{B,p}^\chi$ follows from the fact that $R_p^\chi(\lambda) = V_{2m,p}\circ \tilde R_p(\lambda)\circ V_{0,p}^{-1}$, where we refer to the text preceding Proposition~\ref{P:2.13} for terminology.

Next we note that if $\lambda_0$ belongs to the resolvent set of $\tilde A_{B,p}$, $T_p=(\tilde A_{B,p}-\lambda_0)^{-1}$, and $\mu$ is an eigenvalue of $\tilde A_{B,p}$, then $(\mu-\lambda_0)^{-1}$ is an eigenvalue of the compact operator $T_p$ and the principal subspace of $\tilde A_{B,p}$ corresponding to the eigenvalue $\mu$ has the same multiplicity as the principal subspace of $T_p$ corresponding to the eigenvalue $(\mu-\lambda_0)^{-1}$. Thus the non-zero eigenvalues of $\tilde A_{B,p}$ have finite algebraic multiplicities, and a similar remark holds for $A_{B,p}^\chi$.

Let $\bar u_{2m}\in D(A_{B,p}^\chi)$, where $D(\cdot)$ denotes the domain, and for $\lambda\in\C$ let $(A_{B,p}^\chi-\lambda)\bar u_{2m} = \bar u_0\in W_p^{0,\chi}(\Omega)$. Then we know that $\bar v_{2m} = V_{2m,p}^{-1} \bar u_{2m}\in D(\tilde A_{B,p})$ and $(\tilde A_{B,p}-\lambda)\bar v_{2m} = \bar v_0 = V_{0,p}^{-1}\bar u_0$. Hence if $\{u_j\}_{j=0}^{\ell-1}$ is a chain of length $\ell$ consisting of the eigenvector $u_0$ and the associated vectors $u_j$, $1\le j\le \ell-1$,  corresponding to the eigenvalue $\mu$ of $A_{B,p}^\chi$, that is,
$(A_{B,p}^\chi-\mu) u_0=0$, $(A_{B,p}^\chi-\mu)u_j=u_{j-1}$ for $j\ge 1$, and $(A_{B,p}^\chi-\mu)^\ell u_{\ell-1}=0$, then it follows that $\{v_j\}_{j=0}^{\ell-1}$ is a chain of length $\ell$ consisting of the eigenvector $v_0$ and associated vectors $v_j$, $1\le j\le \ell$, corresponding to the eigenvalue $\mu$ of $\tilde A_{B,p}$, where $v_0 = V_{2m,p}^{-1} u_0$ and $v_j=V_{2m,p}^{-1} u_j$ for $j\ge 1$. Since a similar result holds if we interchange the roles of $A_{B,p}^\chi$ and $\tilde A_{B,p}$, all the assertions of the proposition follow.
\end{proof}

As a consequence of Proposition~\ref{P:4.2}, we are now able to present the main results of this section.

\begin{theorem}
  \label{T:4.3}
  The eigenvalues as well as the principal vectors of $A_{B,p}^\chi$ corresponding to each such eigenvalue are the same for all $p$, $1<p<\infty$.
\end{theorem}

\begin{proof}
  We know from \cite{A1} that the assertion is true when $A_{B,p}^\chi$ is replaced by$\tilde A_{B,p}$. That the assertion is also true for $A_{B,p}^\chi$ follows from Proposition~\ref{P:4.2}.
\end{proof}

\begin{theorem}
  \label{T:4.4}
  Let $\{\mathcal L_k\}_{k=1}^\ell$ denote a family of distinct rays in the complex $\lambda$-plane which emanate from the origin and which divide the $\lambda$-plane into $\ell$ sectors. Suppose in addition that the boundary problem \eqref{E:1.1}, \eqref{E:1.2} is parameter-elliptic along each of the rays $\mathcal L_k$ and that the angle between any two adjacent rays is less than $2m\pi/n$. Then $A_{B,p}^\chi$ has an infinite number of eigenvalues and the corresponding principal vectors are complete in $W_p^{0,\chi}(\Omega)$, $1<p<\infty$.
\end{theorem}

\begin{proof}
  Since $D(A_{B,p})$ is dense in $L_p(\Omega)$, it follows from \eqref{E:3.4} and Proposition~\ref{P:4.1} that the same is true for $D(\tilde A_{B,p})$. Hence we can argue as in \cite[Proof of Theorem~3.2]{A1} to show that the theorem is true when $A_{B,p}^\chi$ is replaced by $\tilde A_{B,p}$. That the assertion is true for $A_{B,p}^\chi$ follows from Propositions~\ref{P:2.12} and \ref{P:4.2}.
\end{proof}

In the following theorem we let $\mathcal L(\theta)$ denote the ray in the complex $\lambda$-plane emanating from the origin and making an angle $\theta$ with the positive real axis.

\begin{theorem}
  \label{T:4.5}
  Suppose the boundary problem \eqref{E:1.2}, \eqref{E:1.2} is parameter-elliptic along each of the rays $\mathcal L(\theta_1)$ and $\mathcal L(\theta_2)$, where $0<\theta_2-\theta_1<\min\{2m\pi/n, 2\pi\}$, but not parameter-elliptic along the ray $\mathcal L(\theta_0)$, where $\theta_1<\theta_0<\theta_2$. Then the sector $\mathcal L^\#$ determined by the inequalities $\theta_1\le \arg\lambda\le\theta_2$ contains an infinite number of eigenvalues of $A_{B,p}^\chi$.
\end{theorem}

\begin{proof}
  In light of Proposition~\ref{P:4.2} and Theorem~\ref{T:4.3}, we need only prove the theorem for $\tilde A_{B,p}$ in place of $A_{B,p}^\chi$ and for $p=2$. Accordingly, we know from \cite{A1} that the assertion is certainly true for $A_{B,2}$ in place of $\tilde A_{B,2}$. On the other hand if we suppose that the assertion is false for $\tilde A_{B,2}$, then we also know from \cite{A1} that there exist positive constants $C^\#$ and $\lambda^\#$ such that the set $\mathcal L_1^\# = \{\lambda\in \mathcal L^\#\,\big|\, |\lambda|\ge\lambda^\#\}$ belongs to the resolvent set of $\tilde A_{B,2}$ and for $\lambda\in\mathcal L_1^\#$ the inequality $|\lambda|\,\|\tilde R_2(\lambda)\|_{L_2(\Omega)\to L_2(\Omega)}\le C^\#$ holds. But since it follows from \cite[Proposition~2.3]{ADF} and Proposition~\ref{P:3.8} that the boundary problem \eqref{E:1.1}, \eqref{E:1.2} is parameter-elliptic in $\mathcal L^\#$, we conclude from \cite[Theorem~2.1]{ADF} that there is a constant $\lambda_0>0$ such that the set $\{\lambda\in\mathcal L^\#\,\big|\, |\lambda|\ge \lambda_0\}$ belongs to the resolvent set of $A_{B,2}$ which is a contradiction. This completes the proof of the proposition.
\end{proof}

Let us now turn to the asymptotic behaviour of the eigenvalues of $A_{B,p}^\chi$. Accordingly for $0<\theta<\pi$ let $\mathcal L_\theta$ denote the closed sector in the complex plane with vertex at the origin determined by the inequalities $\theta\le|\arg\lambda|\le\pi$. Then guided by future considerations we shall henceforth suppose that the sector $\mathcal L$ defined in the text following \eqref{E:1.2} coincides with $\mathcal L_\theta$ and that $\overline{\R_-}$ belongs to the resolvent set of $A_{B,p}^\chi$. We note that there is no loss of generality incurred by these assumptions since they can always be achieved by means of a rotation and a shift in the spectral parameter. Note also from Theorem~\ref{T:3.10} that there are at most a finite number of eigenvalues of $A_{B,p}^\chi$ contained in $\mathcal L_\theta$. Furthermore, we denote the eigenvalues of $A_{B,2}^\chi$ by $\{\lambda_j\}_{j\ge 1}$, where each eigenvalue is counted according to its algebraic multiplicity and arranged so that the $\{|\lambda_j|\}_{j\ge 1}$ form a non-decreasing sequence in $\R_+$. We note of course that the $\lambda_j$ are the eigenvalues of $A_{B,p}^\chi$ and $\tilde A_{B,p}$ for all $p$, $1<p<\infty$.

For $t>0$ let $N(t)$ denote the number of eigenvalues $\{\lambda_j\}_{j\ge 1}$ of $A_{B,p}^\chi$ for which $|\lambda_j|\le t$.

\begin{theorem}
  \label{T:4.6}
  Let the boundary problem \eqref{E:1.1}, \eqref{E:1.2} be parameter-elliptic along every ray emanating from the origin in the complex plane except along $\overline{\R_+}$. Then
  \[ N(t) = d t^{n/(2m)} + o (t^{n/(2m)})\;\text{as }t\to \infty,\text{ where } d = \frac{1}{(2\pi)^n} \int_\Omega dx \int_{\mathring A(x,\xi)<1} d\xi.\]
\end{theorem}

Before turning to the proof of Theorem~\ref{T:4.6}, let us make the following observations. Firstly, under our assumptions we know that $\mathring A(x,\xi)\not=0$ for $\xi\not=0$. Secondly, it follows from Theorem~\ref{T:4.5} that $A_{B,p}^\chi$ has an infinite number of eigenvalues. Furthermore, it follows from Proposition~\ref{P:4.2} and Theorem~\ref{T:4.3} that we need only prove the theorem with $A_{B,p}^\chi$ replaced by $\tilde A_{B,2}$. And in order to achieve this end we turn to the von Neumann-Schatten class of compact operators on $L_2(\Omega)$ (see \cite[Chapters II and III]{GK}).

Let $T$ be a compact operator on $L_2(\Omega)$. Then the non-zero eigenvalues $\{s_j(T)\}_{j\ge 1}$ of the non-negative operator $(T^*T)^{1/2}$, arranged so that $s_1(T)\ge s_2(T)\ge \ldots$, with each eigenvalue repeated according to its multiplicity, are called the singular values of $T$. For $0<q<\infty$, we denote by $\mathscr S_q$ the class of compact operators $T$ for which $\sum_{\ell\ge 1} s_\ell(T)^q<\infty$, and for $q\ge 1$ and $T\in \mathscr S_q$ we let $|T|_q = (\sum_{\ell\ge 1} s_\ell(T)^q)^{1/q}$. Note that $|\cdot|_q$ is a norm on $\mathscr S_q$ and with respect to this norm, $\mathscr S_q$ is a Banach space. Note also that if $q_1<q_2$, then $\mathscr S_{q_1}\subset \mathscr S_{q_2}$. The class $\mathscr S_2$ are the Hilbert-Schmidt operators, that is the class of compact operators $T$ which can be represented as an integral operator:
\begin{equation}
  \label{E:4.2}
  Tf(x) = \int_\Omega K(x,y)f(y)dy\quad\text{for }f\in L_2(\Omega),
\end{equation}
where $K\in L_2(\Omega\times\Omega)$. The operators from $\mathscr S_1$ are the trace class operators, that is, they have the trace
\[ \tr T = \sum_{j\ge 1} \lambda_j(T)\quad\text{for }T\in \mathscr S_1,\]
where $\{\lambda_j(T)\}_{j\ge 1}$ denote the non-zero eigenvalues of $T$, with each eigenvalue repeated according to its algebraic multiplicity and arranged so that their moduli form a non-increasing sequence in $\R_+$, and where the series converges absolutely. Furthermore, for $T\in \mathscr S_1$, we have
\[ |\tr T|\le |T|_1,\]
and $T$ is an integral operator, with the kernel $K(x,y)$ in \eqref{E:4.2} being continuous in $\bar\Omega\times\bar\Omega$, and we also have $\tr T = \int_\Omega K(x,x)dx$. Note that if in Theorem~\ref{T:3.10} we suppose that $\lambda\in\mathcal L$ with $|\lambda|\ge\lambda^0$, then it follows from that theorem and \cite[Subsection 4.2]{ADF} that for any $\epsilon>0$ and $\ell\in\N$ we have
\begin{equation}
  \label{E:4.3}
  \tilde R_2(\lambda) \in \mathscr S_{n(1+\epsilon)/(2m)}\quad\text{and } \tilde R_2(\lambda)^\ell \in \mathscr S_{n(1+\epsilon)/(2m\ell)}.
\end{equation}

\begin{proposition}
  \label{P:4.7}
  If $2m>n$, then put $k=1$, while if $2m\le n$ let $q$ denote the smallest even integer greater than $n/(2m)$ and put $k=q/2$. Then for $\lambda\in\mathcal L_\theta$ with $|\lambda|\ge\lambda^0$, $\tilde R_2(\lambda)^k$ is a Hilbert-Schmidt operator such that in its integral representation its kernel $\tilde K(x,y,\lambda)$ for  $x\in \bar\Omega$, $\lambda\in \mathcal L_{\theta}$, the map  $x\mapsto K(x,\cdot,\lambda),\, \bar\Omega\to L_2(\Omega)$ is continuous  for each $\lambda\in\mathcal L_\theta$ and
  \begin{equation}
    \label{E:4.4}
    \Big( \int_\Omega |\tilde K(x,y,\lambda)|^2 dy\Big)^{1/2} \le C_k |\lambda|^{\frac{n}{4m} -k},
  \end{equation}
  where the constant $C_k$ does not depend upon $x$ and $\lambda$.
\end{proposition}

\begin{proof}
  To begin with, let us mention that the proposition has been proved in \cite[Lemma~2.1, Theorem~5.1, and equation (7.7)]{A2} for the case $2m>n$ and in \cite[Section~5]{ADF} otherwise. However since we wish to refer to the proof of this proposition in the sequel, we shall give a brief outline of the proof given in \cite{ADF}. Accordingly, we note from \eqref{E:4.3} that for   $\lambda\in\mathcal L_\theta$ with $|\lambda|\ge\lambda^0$, $\tilde R_2(\lambda)^k$ is a Hilbert-Schmidt operator on $L_2(\Omega)$ and its kernel is denoted by $\tilde K(x,y,\lambda)$.
We suppose henceforth that $\lambda\in\mathcal L_\theta$ with $|\lambda|\ge\lambda^0$. Then in order to prove the cited assertions, the following facts will be used: (1) if $2<p<\infty$, then $\tilde R_p(\lambda) = \tilde R_2(\lambda)|_{L_p(\Omega)}$, and (2) if $1<p<p_1$, $s\in\N$, and $0<\tau < \frac ns(p^{-1}-p_1^{-1})<1$, then the embedding $W_p^s(\Omega)\to L_{p_1}(\Omega)$ is continuous and for $u\in W_p^s(\Omega)$ we have the estimate $\| u\|_{0,p_1,\Omega}\le C_0 \| u\|_{s,p,\Omega}^\tau \| u\|_{0,p,\Omega}^{1-\tau}$, where the constant $C_0$ does not depend upon $u$. With these facts in mind, let us choose the numbers $\{p_j\}_{j=1}^k$ so that
$ 2=p_1<p_2<\ldots<p_k$, where $p_k>\tfrac{n}{2m}$ and $0<\tau_j = \tfrac{n}{2m}(\tfrac1{p_j} -\tfrac1{p_{j+1}})<1$ for $j=1,\dots,k-1$.

Then we can write $\tilde R_2(\lambda)$, considered  as a mapping from $L_2(\Omega)$ into $W_2^{2m}(\Omega)$, as a product of operators $S_j \tilde R_{p_j}(\lambda)$, $j=1,\dots,k-1$, where $S_j$ is the embedding operator cited above mapping $W_p^{2m}(\Omega)$ into $L_{p_j}(\Omega)$:
\begin{align*}
 L_{p_1}(\Omega)\overset{S_1\tilde R_{p_1}(\lambda)}\longrightarrow & L_{p_2}(\Omega)\overset{S_2\tilde R_{p_2}(\lambda)}\longrightarrow \ldots \\
 &\ldots \longrightarrow
L_{p_{k-1}}(\Omega)\overset{S_{k-1}\tilde R_{p_{k-1}}(\lambda)}\longrightarrow
L_{p_{k }}(\Omega)\overset{\tilde R_{p_{k}}(\lambda)}\longrightarrow W_{p_k}^{2m}(\Omega).
\end{align*}
It follows immediately from the embedding estimate cited above and \eqref{E:4.1} that
\[ \|\tilde R_2(\lambda)\|_{L_2(\Omega)\to W_p^{2m}(\Omega)} \le C|\lambda|^{\frac{n}{2m}-k},\]
where the constant $C$ does not depend upon $\lambda$, and hence we can argue as in \cite[Lemma~2.1]{A2} to establish all the assertions of the proposition.
\end{proof}

\begin{proposition}
  \label{P:4.8}
  For $\lambda\in\mathcal L_\theta$ with $|\lambda|\ge\lambda^0$, $\tilde R_2(\lambda)^q$ is an operator of trace class and
  \[ \big| \tr \tilde R_2(\lambda)^q\big| \le C|\lambda|^{\frac{n}{2m}-q},\]
  where the constant $C$ does not depend upon $\lambda$.
\end{proposition}

\begin{proof}
  If we observe that $\tilde R_2(\lambda)^q = \tilde R_2(\lambda)^k \tilde R_2(\lambda)^k$, then the assertion of the proposition is a immediate consequence of Proposition~\ref{P:4.7} and \cite[Theorems 2.12, 2.18, and 2.19]{A3}.
\end{proof}

We now present a sharpening of Proposition~\ref{P:4.8}.

\begin{proposition}
  \label{P:4.9}
  It is the case that
  \[ \tr \tilde R_2(\lambda) ^q = c_q (-\lambda)^{\frac{n}{2m}-q} + o \big( |\lambda|^{\frac{n}{2m}-q}\big) \text{ uniformly in }\mathcal L_\theta \text{ as }|\lambda|\to\infty,\]
  where $c_q = \int_\Omega c_q(x)dx$, $c_q(x) = (2\pi)^{-n} \int_{\R^n} \frac{d\xi}{(\mathring A(x,\xi)+1)^q}$, and where we assign to $\arg(-\lambda)$ its value in $[-\pi+\theta,\pi-\theta]$.
\end{proposition}

\begin{proof}
  Supposing henceforth that $\lambda\in\mathcal L_\theta$ with $|\lambda|\ge\lambda^0$, we know from \eqref{E:4.1} and \cite[Section~5]{ADF} that $R_2(\lambda)^k$ is also a Hilbert-Schmidt operator, and if we let $K(x,y,\lambda)$ denote its associated kernel, then all the assertions of Proposition~\ref{P:4.7} hold in full force with $\tilde R_2(\lambda)^k$ and $\tilde K(x,y,\lambda)$ replaced by $R_2(\lambda)^k$ and $K(x,y,\lambda)$, respectively.
  Consequently Proposition~\ref{P:4.8} holds in full force with $\tilde R_2(\lambda)$ replaced by $R_2(\lambda)$.

  We are now going to obtain an estimate for $\tr (\tilde R_2(\lambda)^q - R_2(\lambda)^q)$. To this end let us observe that with $q_1,q_2\in\N_0$,
  \begin{align*}
    \tilde R_2(\lambda)^q-R_2(\lambda)^q & = \Big( \sum_{q_1+q_2=k-1} R_2(\lambda)^{q_1}(\tilde R_2(\lambda) -R_2(\lambda) )\tilde R_2(\lambda)^{q_2}\Big)\tilde R_2(\lambda)^k \\
    & + R_2(\lambda)^k \Big( \sum_{q_1+q_2=k-1} R_2(\lambda)^{q_1} (\tilde R_2(\lambda)-R_2(\lambda))\tilde R_2(\lambda)^{q_2}\Big).
  \end{align*}
  Hence if we fix our attention upon a fixed pair $q_1,q_2$ and appeal to Proposition~\ref{P:4.1} and \cite[Section~5]{ADF}, then we can argue as we did in the proof of Proposition~\ref{P:4.7} to show that $R_2(\lambda)^{q_1}(\tilde R_2(\lambda)-R_2(\lambda))\tilde R_2(\lambda)^{q_2}$ is a Hilbert-Schmidt operator, and if we let $K^\dagger (x,y,\lambda)$ denote its associated kernel, then all the assertions of Proposition~\ref{P:4.7} with $\tilde R_2(\lambda)$, $\tilde K(x,y,\lambda)$, and $C_k |\lambda|^{\frac{n}{2m}-k}$ replaced by $R_2(\lambda)^{q_1}(\tilde R_2(\lambda)-R_2(\lambda))\tilde R_2(\lambda)^{q_2}$, $K^\dagger(x,y,\lambda)$, and $C_k |\lambda|^{-1/(2m)}|\lambda|^{n/(4m) - k}$, respectively. In light of this fact we can appeal to \cite[Section~5]{ADF} and argue as we did in the proof of Proposition~\ref{P:4.8} to show that
  \[ \tr\big( \tilde R_2(\lambda)^q - R_2(\lambda)^q\big) \le C |\lambda|^{-\frac1{2m}}|\lambda|^{\frac{n}{2m}-q},\]
  where the constant $C$ does not depend upon $\lambda$. Since $\tr \tilde R_2(\lambda)^q = \tr (R_2(\lambda)^q) + \tr (\tilde R_2(\lambda)^q-R_2(\lambda)^q)$, the assertion of the proposition is an immediate consequence of the forgoing results and \cite[Theorem~5.1]{ADF}.
\end{proof}

\begin{remark}
  \label{R:4.10}
  Referring to Proposition~\ref{P:4.9} we note from \cite{AM} and \cite[pp.~110-111]{S} that $c_q$ can also be written in the form
  \[ c_q = \frac{1}{n(2\pi)^n} b_{\frac{n}{2m}, q} \int_{\Omega} dx \int_{|\eta|=1} \mathring A(x,\eta)^{-\frac{n}{2m}} d\eta,\]
  as well as in the form
  \[ c_q = \frac{1}{(2\pi)^n} b_{\frac{n}{2m},q} \int_\Omega dx\int_{\mathring A(x,\xi)<1} d\xi,\]
  where $b_{\frac{n}{2m},q} = \frac{n}{2m} B(\frac{n}{2m}, q-\frac{n}{2m})$ and $B(\cdot,\cdot)$ denotes the Beta function.
\end{remark}

\begin{proof}
  [Proof of Theorem~\ref{T:4.6}]
  As stated above we need only prove the theorem with $A_{B,p}^\chi$ replaced by $\tilde A_{B,2}$. But the proof for $\tilde A_{B,2}$ follows immediately from Proposition~\ref{P:4.9}, Remark~\ref{R:4.10}, and the arguments used in the proof of \cite[Theorem~6.3]{ADF}.
\end{proof}

\end{document}